\documentclass[fleqn,final]{zzz}
\usepackage{color}
\usepackage{tikz}
\usepackage{amsmath,amsthm,amsbsy,amssymb}
\usepackage{graphicx}
\usepackage{rotating}
\usepackage{fixmath}
\usepackage[pdftex]{hyperref}
\usepackage{soul}

\hypersetup{
 colorlinks = true,
 urlcolor = blue,
 linkcolor = red,
 citecolor = green,
}

\definecolor{Mycolor2}{HTML}{0f4c5c}
\sethlcolor{black}
\newcommand\poro[1]{{\textcolor{purple}{#1}}}

\newcommand\boro[1]{{\textcolor{black}{#1}}}
\newcommand{\hyp}[5]{\,\mbox{}_{#1}F_{#2}\!\left(
\genfrac{}{}{0pt}{}{#3}{#4};#5\right)}

\def\cprime{$'$}

\newtheorem{thm}{Theorem}[section]
\newtheorem{cor}[thm]{Corollary}
\newtheorem{rem}[thm]{Remark}

\newtheorem{prop}[thm]{Proposition}

\makeatletter
\def\eqnarray{\stepcounter{equation}\let\@currentlabel=\theequation
\global\@eqnswtrue
\tabskip\@centering\let\\=\@eqncr
$$\halign to \displaywidth\bgroup\hfil\global\@eqcnt\z@
$\displaystyle\tabskip\z@{##}$&\global\@eqcnt\@ne
\hfil$\displaystyle{{}##{}}$\hfil
&\global\@eqcnt\tw@ $\displaystyle{##}$\hfil
\tabskip\@centering&\llap{##}\tabskip\z@\cr}

\def\endeqnarray{\@@eqncr\egroup
\global\advance\c@equation\m@ne$$\global\@ignoretrue}

\makeatother

\newcommand{\expe}{{\mathrm e}}

\newcommand{\biQ}{{\mathbold{Q}}}

\usepackage{scalerel,url}
\let\svus_
\catcode`_=\active %
\def_{\ifmmode\expandafter\svus\expandafter\bgroup\expandafter\lowerit
\else\svus\fi}
\def\lowerit#1{\ThisStyle{\raisebox{-2\LMpt}{$\SavedStyle#1$}}\egroup}
\makeatletter
 \AtBeginDocument{
 \check@mathfonts
 \fontdimen13\textfont2=3.5pt
 \fontdimen14\textfont2=3.5pt
 \fontdimen16\textfont2=2.5pt
 \fontdimen17\textfont2=2.5pt
 }
\makeatother

\begin{document}

\renewcommand{\PaperNumber}{***}

\FirstPageHeading

\ShortArticleName{On the relation between
Gegenbauer polynomials and Ferrers functions}

\ArticleName{On the relation between Gegenbauer polynomials and 
the Ferrers function of the first kind}

% Names of the authors for the title of the paper
\Author{Howard S. Cohl%Please carefully check the accuracy of 
%names and affiliations. 
$\,^{\dag}$ and 
Roberto S. Costas-Santos$\,^{\S}$
%and Linus Ge $^{3,\dagger}$*
}

\AuthorNameForHeading{H.~S.~Cohl and R.~S.~Costas-Santos}
\Address{$^\dag$ Applied and Computational 
Mathematics Division, National Institute of Standards 
and Tech\-no\-lo\-gy, Mission Viejo, CA 92694, USA
%Address of First Author, Country
\URLaddressD{
\href{http://www.nist.gov/itl/math/msg/howard-s-cohl.cfm}
{http://www.nist.gov/itl/math/msg/howard-s-cohl.cfm}
}
} % Address of First Author
\EmailD{howard.cohl@nist.gov} % E-mail address of First Author

\Address{$^\S$ Dpto. de F\'isica y Matem\'{a}ticas,
Universidad de Alcal\'{a},
c.p. 28871, Alcal\'{a} de Henares, Spain} 
% Address of First Author
\URLaddressD{
\href{http://www.rscosan.com}
{http://www.rscosan.com}
}
\EmailD{rscosa@gmail.com} % E-mail address of First Author

%\Address{$^{\S\S}$ Department of Mathematics,
%University of Rochester, Rochester, NY 14627, USA
%% Address of First Author
%}
%\EmailD{random9483@gmail.com} % E-mail address of First Author

\ArticleDates{Received \today~in final form ????; Published online ????}

\Abstract{Using the direct relation between the 
Gegenbauer polynomials and 
the Ferrers function of the first kind, we compute
interrelations between certain Jacobi polynomials, 
Meixner 
polynomials, and the Ferrers function of the first kind.
We then compute 
Rodrigues-type and orthogonality relations for 
Ferrers functions of the first and second kinds. 
In the remainder of the paper using the relation 
between Gegenbauer polynomials and the Ferrers 
function of the first kind we derive connection 
and linearization relations, some definite integral 
and series expansions, some asymptotic 
expansions of Mehler-Heine type,
Christoffel-Darboux summation formulas, 
and infinite series closure
relations (Dirac delta distribution).
}
%\moro{[RCS: Read it, I rearanged the order of some sentences.]}
%``Symmetry, Integrability and Geometry: Methods and Applications''.}
%Please type here List of Keywords for your article separated 
%by semicolon.
% Keywords:

\Classification{33C05; 33C15; 33C45;42C10}
%{??????} % e.g. 35A30; 81Q05
%For 2010 Mathematics Subject Classification see
%http://www.ams.org/mathscinet/msc/msc2010.html

\section{Introduction}

%\poro{\bf [HSC: Define generalized hypergeometric sqries.
%%Define the various functions including Jacobi, Gegenbauer,
%ermite, polynomials. Also, Ferrers functions of the first
%and second kinds, and associated Legendre functions
%of the first and second kinds and some interrelations.]}

\boro{The generalized hypergeometric
function \cite[Chapter 16]{NIST:DLMF} is defined by the 
infinite series \cite[(16.2.1)]{NIST:DLMF}
\begin{equation} \label{genhyp}
\hyp{r}{s}{a_1,\ldots,a_r}
{b_1,\ldots,b_s}{z}
:=
\sum_{k=0}^\infty
\frac{(a_1,\ldots,a_r)_k}
{(b_1,\ldots,b_s)_k}
\frac{z^k}{k!},
\end{equation}}
\boro{where $b_i\not \in -\mathbb N_0$, for $i=1, \dots, s$;}
and elsewhere by analytic continuation.
\boro{One interesting limit
which we will use below is
\cite[(1.4.4)]{Koekoeketal}
\begin{equation}
\lim_{\lambda\to\infty}
\hyp{r}{s}{a_1,\ldots,a_r}
{b_1,\ldots,b_{s-1},\lambda b_s}
{\lambda z}
=\hyp{r}{s-1}{a_1,\ldots,a_r}
{b_1,\ldots,b_{s-1}}
{\frac{z}{b_s}}.
\label{limgenhyp}
\end{equation}
}
The Pochhammer symbol for $a\in\mathbb C$,
$n\in\mathbb N_0$ 
is given by \cite[(5.2.4-5)]{NIST:DLMF}
\[
(a)_n:=(a)(a+1)\cdots(a+n-1)
.
\]
The gamma function \cite[Chapter 5]{NIST:DLMF} is related to 
the Pochhammer symbol, namely
for $a\in\mathbb C\setminus-{\mathbb N}_0$, one
 has
\[
(a)_n=\frac{\Gamma(a+n)}{\Gamma(a)},
\]
\boro{which allows one to extend the definition to non-positive 
integer values of $n$.}
We will also use the common notational product convention, e.g.,
\[
(a_1,\ldots,a_r)_k:=(a_1)_k(a_2)_k\cdots(a_r)_k.
\]

\boro{
The special case of $r=0$, $s=1$ in \eqref{genhyp} is 
connected with the
the Bessel function of the first kind
\cite[(10.16.9)]{NIST:DLMF}
\begin{equation}
J_\lambda(z):=\frac{1}{\Gamma(\lambda+1)}
\hyp01{-}{\lambda+1}{-\frac{z^2}{4}}.
\label{Besseldef}
\end{equation}
Furthermore, the Hermite polynomials, related
to the $r=2$, $s=0$ case, can be defined
by 
\cite[(9.15.1)]{Koekoeketal}
\begin{equation}
\label{Hermite}
H_n(x):=(2x)^n\hyp20{-\frac{n}{2},\frac{1-n}{2}}{-}{-\frac{1}{x^2}}.
\end{equation}
The Hermite polynomials are examples of
generalized hypergeometric functions
which are terminating, that is the
infinite series \eqref{genhyp}
truncates at a finite
value of $k$. In fact, all of the 
polynomials discussed in this 
paper will be defined in terms of
a terminating generalized hypergeometric
series.
}

\boro{
The special case of $r=2$, $s=1$ is
referred to as the Gauss hypergeometric
series \cite[Chapter 15]{NIST:DLMF}.
Many of the functions used in this paper are defined 
in terms of the Gauss hypergeometric function. 
Therefore in this manuscript, we will require properties 
for this function.
A useful asymptotic expansion for Gauss hypergeometric
series is given as follows.
\cite[(15.12.2)]{NIST:DLMF},
since $a$ or $b\in- \mathbb N_0$, \boro{for example, if $a=-m$, then}
\begin{equation}
\hyp21{a,b}{c}{z}
=\sum_{k=0}^{m-1}
\boro{\frac{(a,b)_k}{(c)_k}\frac{z^k}{k!}}
+{\mathcal O}\left(
\frac{1}{c^m}\right),
\label{DLMFhypasymp}
\end{equation}
as $c\to\infty$.
}

\boro{
The Jacobi polynomial is defined as \cite[(18.5.7)]{NIST:DLMF}
\begin{equation}
\label{Jacobipolydef}
P_n^{(\alpha,\beta)}(x):=
\frac{(\alpha+1)_n}{n!}
\hyp21{-n,n+\alpha+\beta+1}{\alpha+1}{\frac{1-x}{2}},
\end{equation}
and the Gegenbauer function 
is defined as \cite[(15.9.15)]{NIST:DLMF}
\begin{equation}\label{Gegenbauerfuncdef}
C_\alpha^\lambda(\boro{z}):=\frac{\Gamma(2\lambda+\alpha)}
{\Gamma(2\lambda)\Gamma(\alpha+1)}
\hyp21{-\alpha,2\lambda+\alpha}{\lambda+\frac12}{\frac{1-z}{2}},
\end{equation}
and the classical orthogonal Gegenbauer (ultraspherical) 
polynomial is given when $\alpha\in \mathbb N_0$, which makes 
the Gauss hypergeometric function terminating.
Note that the Gegenbauer polynomial can be given in terms
of the symmetric Jacobi polynomial as \cite[(18.7.1)]{NIST:DLMF}
\begin{equation}\label{GegJac}
C_n^\lambda(x)=\frac{(2\lambda)_n}{(\lambda+\frac12)_n}
P_n^{(\lambda-\frac12,\lambda-\frac12)}(x).
\end{equation}
}
\boro{
Furthermore, there is also
\cite[(18.7.15)]{NIST:DLMF}
\begin{equation}
\label{Geg2nJac}
C_{2n}^\lambda(x)=
\frac{(\lambda)_n}{(\frac12)_n}
P_n^{(\lambda-\frac12,-\frac12)}
(2x^2-1),
\end{equation}
and \cite[(18.7.16)]{NIST:DLMF}
\begin{equation}
\label{Geg2npJac}
C_{2n+1}^\lambda(x)=
\frac{(\lambda)_{n+1}}{(\frac12)_{n+1}}
xP_n^{(\lambda-\frac12,\frac12)}
(2x^2-1).
\end{equation}
}

\boro{
Ferrers functions and Legendre functions are given
in terms of Gauss hypergeometric functions which satisfy
both linear and quadratic transformations. There are many
such transformations and therefore there are many hypergeometric
representations for the Gauss hypergeometric function.
The Ferrers function of the first kind 
\boro{(associated Legendre
function of the first kind on-the-cut)}
\boro{${\sf P}_\nu^{-\mu}:(-1,1)\to\mathbb C$}
can for instance be defined
as \cite[(14.3.1)]{NIST:DLMF}
\begin{equation}
{\sf P}_\nu^{-\mu}(x)=
\frac{1}{\Gamma(\mu+1)}
\left(\frac{1-x}{1+x}\right)^{\frac{\mu}{2}}
\hyp21{-\nu,\nu+1}{\mu+1}{\frac{1-x}{2}}.
\label{FerPdefn}
\end{equation}
\boro{
The Ferrers function of the second kind
${\sf Q}_\nu^\mu:(-1,1)\to\mathbb C$ 
can also be defined
in terms of Gauss hypergeometric representations. 
See for instance the recent publication \cite{Cohletal2021} where all Gauss 
hypergeometric representations of the Ferrers functions 
of the second kind are given.}
\boro{The Ferrers function of the second kind (associated 
Legendre function of the second kind on-the-cut) 
for instance, is given in \cite[(14.3.2)]{NIST:DLMF}
\begin{eqnarray}
%&&\hspace{-1cm}
{\sf Q}_\nu^\mu(x):=\frac{\pi}{2\sin(\pi\mu)}&\Biggl(
&\frac{\cos(\pi\mu)}{\Gamma(1-\mu)}\left(\frac{1+x}
{1-x}\right)^{\frac{\mu}{2}}\hyp21{-\nu,\nu+1}{1-\mu}
{\frac{1-x}{2}}\nonumber\\
%&&\hspace{2cm}
&-&\frac{\Gamma(\nu+\mu+1)}{\Gamma(\nu-\mu+1)}
\left(\frac{1-x}{1+x}\right)^{\frac{\mu}{2}}
\frac{1}{\Gamma(1+\mu)}
\hyp21{-\nu,\nu+1}{1+\mu}{\frac{1-x}{2}}\Biggr),
\label{FerQdefn}
\end{eqnarray}
where $\mu\not\in\mathbb Z$.
However, ${\sf Q}_\nu^\mu$ can be analytically continued 
for $\mu\in\mathbb Z$ which is demonstrated by 
\cite[(14.3.12)]{NIST:DLMF}.}
\boro{Note that in this paper we will often
indicate that the domain of the Ferrers function of the first
and second kinds is $(-1,1)$. But it should
be emphasized that, in general and for specific formulas, this 
region can be analytically continued to a much larger region 
in the complex plane.}
It will often be convenient
to express the Ferrers function of
the first kind in terms of a 
Gauss hypergeometric series in a different way.
One powerful such series representation is given as follows.
}
\boro{
\begin{thm}
Let $\nu,\mu\in\mathbb C$, $x\in(-1,1)$. Then
\begin{equation} \label{FerPnummu}
{\sf P}_{\nu}^{-\mu}(x)=\frac{x^{\nu-\mu}(1-x^2)^{\frac{\mu}{2}}}
{2^\mu\Gamma(\mu+1)}\hyp21{\frac{\mu-\nu}{2},\frac{\mu-\nu+1}{2}}
{\mu+1}{1-\frac{1}{x^2}}.
\end{equation}
\end{thm}
}
\boro{
\begin{proof}
Applying the connection relation \cite[(14.9.1)]{NIST:DLMF}
\[
{\sf P}_{\nu}^{-\mu}(x)=-\frac{2}
{\pi}\csc(\pi\mu)
\frac{\Gamma(\nu-\mu+1)}
{\Gamma(\nu+\mu+1)}
{\sf Q}_\nu^\mu(x)
+\frac{2}{\pi}\cot(\pi\mu)
{\sf Q}_\nu^{-\mu}(x),
\]
to the Gauss hypergeometric representation of
the Ferrers function of the second
kind \cite[(49)]{CohlCostasSantos2020}
\begin{eqnarray*}
&&\hspace{-0.5cm}
{\sf Q}_\nu^\mu(x)=\frac{2^{\mu-1}\cos(\pi\mu)}
{(1-x^2)^{\frac{\mu}{2}}}
\Gamma(\mu)x^{\nu+\mu}\hyp21{\frac{-\nu-\mu}{2}, 
\frac{-\nu-\mu+1}{2}}{1-\mu}{1-\frac{1}{x^2}}
\nonumber\\ 
&&\hspace{3cm}
+\frac{\Gamma(\nu+\mu+1)\Gamma(-\mu)}{2^{\mu+1}
\Gamma(\nu-\mu+1)}(1-x^2)^{\frac{\mu}{2}} x^{\nu-\mu}
\hyp21{\frac{\mu-\nu}{2},\frac{\mu-\nu+1}{2}}{\mu+1}
{1-\frac{1}{x^2}},
%\label{repFerQonem1ox2}
\end{eqnarray*} 
and after some simplification, one completes the proof.
\end{proof}
}

\boro{One special case which we will encounter frequently 
below is as follows.}

%\moro{[RCS: the next result is a corollary, not a theorem. 
%$\mu\to \lambda$, and $\nu\to \lambda+n$]}
\boro{
\begin{cor}
Let $n\in \mathbb N_0$, $x\in(-1,1)$. Then
\begin{equation} \label{FerPnlammlan}
{\sf P}_{n+\lambda}^{-\lambda}(x)=
\dfrac{x^n(1-x^2)^{\frac{\lambda}{2}}}{2^\lambda\Gamma(\lambda+1)}
\hyp21{-\frac{n}{2},\frac{1-n}{2}}{1+\lambda}{1-\dfrac{1}{x^2}}
=\boro{\frac{n!}{(2\lambda+1)_n}
{\sf P}_{\lambda}^{-\lambda}(x) C_n^
{\lambda+\frac 12}(x)}.
\end{equation}
\end{cor}
}
\boro{
\begin{proof}
\boro{Letting $\mu\mapsto \lambda$, and $\nu\mapsto 
\lambda+n$ in \eqref{FerPnummu} proves
the first relation. The second identity holds due to the 
definition of the Gegenbauer function \eqref{Gegenbauerfuncdef}
for $\alpha\in\mathbb N_0$ and taking into account 
\cite[(14.5.18)]{NIST:DLMF}
\[
{\sf P}_{\lambda}^{-\lambda}(x)=
\frac{(1-x^2)^{\frac \lambda2}}{2^{\lambda}\Gamma(\lambda+1)}.
\]
This completes the proof.}
\end{proof}
}

\boro{
Some other special cases which are interesting 
and will prove to be useful 
include \cite[(14.5.11)]{NIST:DLMF}
\begin{eqnarray}
&&\hspace{-7.5cm}{\sf P}_\nu^{\frac12}(\cos\theta)=
\sqrt{\frac{2}{\pi\sin\theta}}\cos((\nu+\tfrac12)
\theta),\label{FerPhalf}\\
&&\hspace{-7.5cm}{\sf P}_\nu^{-\frac12}(\cos\theta)=
\sqrt{\frac{2}{\pi\sin\theta}}
\frac{\sin((\nu+\tfrac12)\theta)}{\nu+\tfrac12}.
\label{FerPmhalf}
\end{eqnarray}
}

\medskip
%\boro{
\boro{
Using the connection relation
\cite[(14.9.1)]{NIST:DLMF}, the Ferrers function of the first kind 
can be expressed in terms of the 
Ferrers function of the second kind, namely
\begin{equation}
{\sf P}_{k+n-\frac12}^{\frac12-n}(x)
=\frac{2(-1)^nk!}{\pi(2n+k-1)!}
{\sf Q}_{k+n-\frac12}^{n-\frac12}(x).
\label{FerPFerQ}
\end{equation}
}
The Ferrers function of the second kind has an interesting
trigonometric special case for $\mu=\frac12$, namely 
\cite[(14.5.13)]{NIST:DLMF}
\begin{equation} \label{Qhalf}
{\sf Q}_\nu^{\frac12}(\cos\theta)
=-\sqrt{\dfrac{\pi}{2\sin\theta}}
\sin((\nu+\tfrac12)\theta),
\end{equation}
and for $\mu=-\frac12$, namely 
\cite[(14.5.14)]{NIST:DLMF}
\begin{equation} \label{Qmhalf}
{\sf Q}_\nu^{-\frac12}(\cos\theta)
=\sqrt{\dfrac{\pi}{2\sin\theta}}
\dfrac{\cos((\nu+\tfrac12)\theta)}{\nu+\frac12}.
\end{equation}
%}

\boro{
The Ferrers function of the first kind is related to 
the Gegenbauer function \cite[(14.3.21)]{NIST:DLMF}
\begin{eqnarray} \label{relFerPGeg}
%\hspace{-5cm}
{\sf P}_\nu^{-\mu}(x)&=&\frac{\Gamma(2\mu+1)\Gamma(\nu-\mu+1)}
{2^\mu\Gamma(\nu+\mu+1)\Gamma(\mu+1)}(1-x^2)^{\frac{\mu}{2}}
C_{\nu-\mu}^{\mu+\frac12}(x)
\\
%\hspace{-3.7cm}
&=&\boro{\frac{\Gamma(2\mu+1)\Gamma(\nu-\mu+1)}{\Gamma(\nu+\mu+1)}
{\sf P}_\mu^{-\mu}(x)C_{\nu-\mu}^{\mu+\frac12}(x)},
\nonumber
\end{eqnarray}
where $2\mu+1$, $\nu-\mu+1\not\in- \mathbb N_0$.
Equivalently 
\begin{eqnarray} \label{ClammuFerP}
%&&\hspace{-7cm}
C_\lambda^\mu(x)=\dfrac{\sqrt{\pi}\,\Gamma(2\mu+\lambda)}
{2^{\mu-\frac12}\Gamma(\mu)\Gamma(\lambda+1)}
\frac{{\sf P}_{\lambda+\mu-\frac12}^{\frac12-\mu}(x)}
{(1-x^2)^{\frac{\mu}{2}-\frac14}}
%\\ &&\hspace{-5.9cm}
&=&\boro{\frac{\Gamma(2\mu+\lambda)}
{\Gamma(2\mu)\Gamma(\lambda+1)}\dfrac{{\sf 
P}_{\lambda+\mu-\frac 12}^{\frac 12-\mu}(x)}
{{\sf P}_{\mu-\frac 12}^{\frac12-\mu}(x)}}%\nonumber
, 
\end{eqnarray}
or with $\lambda=n\in \mathbb N_0$, then
\begin{equation} \label{CnmuFerP}
C_n^\mu(x)=\frac{\sqrt{\pi}\,\Gamma(2\mu+n)}
{2^{\mu-\frac12}\Gamma(\mu)n!}
\frac{{\sf P}_{n+\mu-\frac12}^{\frac12-\mu}(x)}
{(1-x^2)^{\frac{\mu}{2}-\frac14}}.
\end{equation}
}

\boro{
In order to compute Sobolev orthogonality
which arises from Gegenbauer polynomials 
in terms of Ferrers functions we will require
several particular interrelations between 
Gegenbauer polynomials in terms of
Ferrers functions.
}
\boro{
\begin{thm}
Let $x\in(-1,1)$, $n\in{\mathbb N}_0$, $N\in{\mathbb N}$ such that
$n\le 2N-1$. Then
\begin{eqnarray}
&&\hspace{-1cm}C_n^{\frac12-N}(x)=
\label{Sob1PP}
\frac{2^{N-1}}{\sqrt{\pi}}
\Gamma(N+\tfrac12)(1-x^2)^{\frac{N}{2}}
\left({\sf P}_{n-N}^{-N}(x)
+(-1)^n{\sf P}_{n-N}^{-N}(-x)\right)\\
&&\hspace{0.65cm}=\frac{2^{N}}{\sqrt{\pi}}
\Gamma(N+\tfrac12)(1-x^2)^{\frac{N}{2}}
\left(
{\sf P}_{n-N}^{-N}(x)
+\frac{(-1)^{n+N}}{n!(2N-n-1)!}
{\sf Q}_{n-N}^N(x)\right).
\label{Sob1PQ}
\end{eqnarray}
\end{thm}
}
\boro{
\begin{proof} 
The connection relation \cite[(14.9.7)]{NIST:DLMF} produces
\begin{equation}
{\sf P}_{n+\lambda-\frac12}^{\frac12-\lambda}(x)
=\frac{n!}{2\sin(\pi\lambda)\Gamma(2\lambda+n)}
\left({\sf P}_{n+\lambda-\frac12}^{\lambda-\frac12}(x)+
(-1)^n{\sf P}_{n+\lambda-\frac12}^{\lambda-\frac12}(-x)\right).
\label{Pconmx}
\end{equation}
Applying \eqref{Pconmx} to \eqref{CnmuFerP} 
and substituting $\lambda=\frac12-N$ produces
\eqref{Sob1PP}. The application of 
the connection relation \cite[(14.9.9)]{NIST:DLMF} to 
the Ferrers function of the first kind 
with argument $-x$ in \eqref{Sob1PP} with
simplification produces \eqref{Sob1PQ}.
This completes the proof.
\end{proof}
}

\boro{
\begin{thm}
Let $x\in(-1,1)$, $n\in{\mathbb N}_0$, $N\in{\mathbb N}$ such that
$n\ge 2N$. Then,
\begin{eqnarray}
&&\hspace{-1cm}C_n^{\frac12-N}(x)=
\label{Sob2P}
\frac{2^{N}}{\sqrt{\pi}}
\Gamma(N+\tfrac12)\frac{(n-2N)!}{n!}(1-x^2)^{\frac{N}{2}}
{\sf P}_{n-N}^{N}(x).
\end{eqnarray}
\end{thm}
}
\boro{
\begin{proof}
Starting with \eqref{CnmuFerP} and 
substituting $\lambda=\frac12-N$
completes the proof.
\end{proof}
}

\boro{
Here we give interrelations between certain
Jacobi polynomials and the Ferrers function
of the first kind from Gegenbauer polynomials
using a quadratic transformation.}

\boro{
\begin{cor}
Let $n\in \mathbb N_0$, $\lambda\in\mathbb C$.
Then 
\begin{equation}
P_n^{(\lambda-\frac12,-\frac12)}
(2x^2-1)=
\frac{2^\lambda\Gamma(\lambda+n+1)}
{n!(1-x^2)^{\frac{\lambda}{2}}}
{\sf P}_{2n+\lambda}^{-\lambda}(x).
\end{equation}
\end{cor}
}
\boro{
\begin{proof}
Start with \eqref{Geg2nJac}
then inserting \eqref{CnmuFerP}, 
and setting $\lambda\mapsto
\lambda+\frac12$ 
and after simplification
completes the proof.
\end{proof}
}

\boro{
\begin{cor}
Let $n\in \mathbb N_0$, $\lambda\in\mathbb C$. Then
\begin{equation}
P_n^{(\lambda,\frac12)}(2x^2-1)=
\frac{2^{\lambda}\Gamma(\lambda+n+1)}{n!x(1-x^2)^{\frac
{\lambda}{2}}}{\sf P}_{2n+\lambda+1}^{-\lambda}(x).
\end{equation}
\end{cor}
}

\boro{
\begin{proof}
Start with
\eqref{Geg2npJac}
then inserting \eqref{CnmuFerP}, 
and setting
$\lambda\mapsto\lambda+\frac12$, and after simplification 
completes the proof.
\end{proof}
}

\boro{
Below, the following related functions will be used.
The associated Legendre function of
the first kind 
$P_\nu^\mu:\mathbb C\setminus(-\infty,1]\to\mathbb C$
is defined as \cite[(14.3.6)]{NIST:DLMF}
\begin{equation}
P_\nu^{-\mu}(z)=
\frac{(z^2-1)^{\frac{\mu}{2}}}
{2^\mu\Gamma(\mu+1)}
\hyp21{\nu+\mu+1,-\nu+\mu}{1+\mu}
{\frac{1-z}{2}}.
\label{EulerLegP}
\end{equation}
}

\noindent\boro{Note that the Legendre polynomials 
are given by \cite[(14.7.1), (18.7.9)]{NIST:DLMF}
\begin{equation}
P_n(x)={\sf P}_n^{0}(x)=P_n^0(x)=C_n^{\frac12}(x).
\label{Legpoly}
\end{equation}
}
\boro{
\begin{rem}
The Chebyshev polynomials
of the first kind \cite[(9.8.35)]{Koekoeketal} can be given by
\begin{equation}
\label{Cheby}
T_n(\cos\theta):=\cos(n\theta)
=\sqrt{\frac{\pi\sin\theta}{2}}
{\sf P}_{n-\frac12}^\frac12(\cos\theta)
=n\sqrt{
\frac{2\sin\theta}{\pi}}
{\sf Q}_{n-\frac12}^{-\frac12}(\cos\theta),
\end{equation}
and the Chebyshev polynomials of the second kind
\cite[(9.8.36)]{Koekoeketal} can be given by
\begin{equation}
\label{ChebyU}
U_n(\cos\theta):=\frac{\sin((n+1)\theta)}{\sin\theta}=
(n+1)\sqrt{\frac{\pi}{2\sin\theta}}
{\sf P}_{n+\frac12}^{-\frac12}(\cos\theta)=
-\sqrt{
\frac{2}{\pi\sin\theta}}
{\sf Q}_{n+\frac12}^{\frac12}(\cos\theta),
\end{equation}
where we have used \eqref{FerPhalf}, \eqref{FerPmhalf}, 
\eqref{Qhalf}, \eqref{Qmhalf}.
See Remark \ref{trigorth} below with
a concrete orthogonality example where the prefactor
$n$ cancels for the Chebyshev polynomials 
of the first kind, even for $n=0$.
\end{rem}
}

\boro{
The associated Legendre function of the second
kind $\biQ_\nu^\mu:\mathbb C\setminus(-\infty,1]\to\mathbb C$
can be defined in terms of the Gauss hypergeometric
function %\soro{for $\nu+\mu\notin-\N$, }
as \cite[%(14.3.7)
(14.3.10) and Section 14.21]{NIST:DLMF}
\begin{equation}
\biQ_\nu^\mu(z):=\frac{\sqrt{\pi}%\expe^{i\pi\mu}\Gamma(\nu+\mu+1)
(z^2-1)^{\mu/2}}
{2^{\nu+1}\Gamma(\nu+\frac32)z^{\nu+\mu+1}}
\hyp21{\frac{\nu+\mu+1}{2},\frac{\nu+\mu+2}{2}}{
\nu+\frac32}{\frac{1}{z^2}},
\label{Qdefn}
\end{equation}
for $|z|>1$ and, by analytic continuation of the Gauss hypergeometric
function, elsewhere on $z\in\mathbb C\setminus(-\infty,1]$.
The normalized notation $\biQ_\nu^\mu(z)$
is due to Olver \cite[p.~178]{Olver}
and is defined in terms of the more
commonly appearing Hobson notation for
the associated Legendre function of the
second kind $Q_\nu^\mu(z)$ as
\cite[(14.3.10)]{NIST:DLMF}
\[
Q_\nu^\mu(z)=:
\expe^{i\pi\mu}\Gamma(\nu+\mu+1)\biQ_\nu^\mu(z).
\]
}

\boro{
The Meixner polynomials are defined by
\cite[(9.10.1)]{Koekoeketal}
\[
M_n(x;\beta,c):=\hyp21{-n,-x}{\beta}{1-\frac{1}{c}},
\]
where $n\in\mathbb N_0$, $c\in \mathbb C
\setminus\{0,1\}$.
}
\noindent \boro{
Specialized Meixner polynomials are related to the 
Ferrers function of the first kind.
}
\boro{
\begin{cor}
Let $n\in{\mathbb N}_0$, $\lambda\in{\mathbb C}$, $x\in(-1,1)$. Then
\begin{equation}\label{jacmeixrel}
{\sf P}_{n+\lambda}^{-\lambda}(x)=
\frac{(1-x^2)^{\frac{\lambda}{2}}}{2^\lambda\Gamma(\lambda+1)}
M_n\left(-2\lambda-n-1;\lambda+1,\frac{2}{1+x}\right).
\end{equation}
\end{cor}
}
\boro{
\begin{proof}
Start with the relation between the Meixner polynomials 
and the Jacobi polynomials \cite[p.~236]{Koekoeketal}
\begin{equation*}%\label{jacmeixrel}
M_n(x;\beta,c)=\frac{n!}{(\beta)_n}
P_n^{(\beta-1,-\beta-n-x)}\left(\frac{2-c}{c}\right),
\end{equation*}
and let $x=1-n-2\beta$. This converts the Jacobi 
polynomial to symmetric form and then using
\eqref{GegJac}, one obtains
\[
C_n^\lambda(\omega)=\frac{(2\lambda)_n}
{n!}M_n\left(-2\lambda-n;\lambda+\tfrac12,\frac{2}{1+\omega}\right).
\]
Finally, using \eqref{CnmuFerP}, completes the proof.
%\[
%{\sf P}_{n+\lambda}^{-\lambda}(\omega)
%=\frac{(1-\omega^2)^{\frac{\lambda}{2}}}{2^\lambda 
%\Gamma(\lambda+1)}
%M_n\left(-2\lambda-n-1;\lambda+1,\frac{2}{1+\omega}\rig%ht).
%\]
\end{proof}
}
\subsection{\boro{Rodrigues-type relations%
%for the Ferrers functions
}}

\boro{
From the Gegenbauer polynomial Rodrigues-type relations, we can 
derive Rodrigues-type relations for the Ferrers function
of the first {\boro and second kinds}.
}

\boro{
\begin{cor}
Let $x\in(-1,1)$, $\mu\in\mathbb C$, 
%$n,l\in \mathbb N_0$ such that $n\ge l$. 
$n\in \mathbb N_0$. 
Then,
\begin{equation}
\label{RodPFer}
{\sf P}_{n+\mu}^{-\mu}(x)
=\frac{(-1)^n}{2^{\mu+n}\Gamma(\mu+n+1)(1-x^2)^{\frac{\mu}{2}}}
\left(\frac{\mathrm d}{{\mathrm d}x}\right)^n
\left[\big(1-x^2\big)^{\mu+n}\right].
\end{equation}
%or in terms of orthogonal Ferrers functions
%of the first kind,
%\begin{equation}
%(1-x^2)^{\frac{\mu+l}{2}}{\sf P}_{n+\mu}^{-\mu-l}(x)
%=\frac{(-1)^{n-l}}{2^{\mu+n}\Gamma(\mu+n)}
%\left(\frac{\mathrm d}{{\mathrm d}x}\right)^{n-l}
%\left[(1-x^2)^{\mu+n}\right].
%\end{equation}
\end{cor}
}

\boro{
\begin{proof}
Start with \eqref{relFerPGeg}, setting $\nu\to \mu+n$ and by 
using the Rodrigues-type formula for the Gegenbauer polynomials 
\cite[(18.5.5)]{NIST:DLMF}
\[
C_n^\mu(x)=\dfrac{(-1)^n(2\mu)_n}{(\mu+\frac 12)_n 2^n 
n!(1-x^2)^{\mu-\frac 12}}\left(\frac{\mathrm d}
{{\mathrm d}x}\right)^n\left[\big(1-x^2\big)^{\mu+n-\frac 12}\right],
\]
and after straightforward calculation completes the proof.
\end{proof}
}

\noindent 
\boro{Similarly we have a Rodrigues-type relation for the 
Ferrers function of the second kind.}

\boro{
\begin{cor}
Let $x\in(-1,1)$, $\mu\in\mathbb C$, 
%$n,l\in \mathbb N_0$ such that $n\ge l$. 
$n\in \mathbb N_0$. 
Then,
\begin{equation}
{\sf Q}_{k+n-\frac12}^{n-\frac12}(x)
=\frac{\pi(-1)^{n+k}(2n+k-1)!}
{2^{k+n+\frac12}k!\Gamma(k+n+\frac12)\big(1-x^2
\big)^{\frac{n}{2}-\frac14}}
\left(\frac{\mathrm d}{{\mathrm d}x}\right)^k
\left[\big(1-x^2\big)^{k+n-\frac12}\right].
\end{equation}
\end{cor}
}
\boro{
\begin{proof}
Start with \eqref{RodPFer}, let $n\mapsto k$, 
$\mu\mapsto{n-\frac12}$ then use \eqref{FerPFerQ}, and 
after simplification, this completes 
the proof.
\end{proof}
}
%%%%%%%%%%%%%%%%%%%%%%%%%%%%%%%%%%%%%%%%%%%%%%%%%%%%%%%%%%%%%%%%%%%%%
\section{\boro{Orthogonality relations}}
%%%%%%%%%%%%%%%%%%%%%%%%%%%%%%%%%%%%%%%%%%%%%%%%%%%%%%%%%%%%%%%%%%%%%
In this section, we derive several new forms of orthogonality for 
Ferrers functions of the first and second kinds
\boro{
which are equivalent to the orthogonality of Gegenbauer 
(ultraspherical) polynomials. Due to this equivalence, 
it is surprising that these orthogonality relations 
have \boro{not} been noticed previously}. 
We will use these orthogonality conditions to verify 
our derived eigenfunctions for the Laplace-Beltrami 
operator on the $d$-dimensional $R$-radius hypersphere 
$\mathbb S_R^d$.
For detailed information about the special functions 
we use below, namely the gamma function
$\Gamma:\mathbb C\setminus- \mathbb N_0\to\mathbb C$, Ferrers 
functions ${\sf P}_\nu^\mu$, 
${\sf Q}_\nu^\mu$%
%\hl{$:(-1,1)\to\mathbb C$}
, and the Gegenbauer polynomials $C_n^\mu$,
%\hl{$C_n^\mu:\mathbb C\to\mathbb C$}, 
and their properties, 
see \cite[Chapters 5, 14, 18]{NIST:DLMF}.
%%%%%%%%%%%%%%%%%%%%%%%%%%%%%%%%%%%%%%%%%%%%%%%%%%%%%%%%%%%%%%%%%%%%%
\subsection{Continuous orthogonality from Gegenbauer polynomials}
%%%%%%%%%%%%%%%%%%%%%%%%%%%%%%%%%%%%%%%%%%%%%%%%%%%%%%%%%%%%%%%%%%%%%
\begin{thm} \label{mainorthogthm}
\noindent Let $\mu\in\mathbb C$, 
%$k,k',l\in\mathbb Z$, $k\ge l$, $\alpha+k+\frac12\ne 0$,
\boro{$k,k'\in\mathbb N$, $\mu+k+\frac12\ne 0$, $\Re\mu>-1$.}
Then,
%\begin{equation}
%\int_{-1}^{1}{\sf P}_{k+\alpha}^{-l-\alpha}(x) 
%{\sf P}_{k'+\alpha}^{-l-\alpha}(x){\mathrm d}x
%=\frac{(k-l)!\,\delta_{k,k'}}{\Gamma(2\alpha+k+l+1)
%(\alpha+k+\frac12)}. 
%\label{orthogonalityPmk}
%\end{equation}
\boro{
\begin{equation} \label{orthogonalityPmk}
\int_{-1}^{1}{\sf P}_{k+\mu}^{-\mu}(x){\sf P}_{k'+\mu}^{-\mu}
(x)\,{\mathrm d}x=\dfrac{k!}{\Gamma(2\mu+k+1)
(\mu+k+\frac12)}\,\delta_{k,k'}. 
\end{equation}
}
\end{thm}
\begin{proof}
Combining \cite[(14.3.21)]{NIST:DLMF}
\eqref{relFerPGeg}
and orthogonality for the Gegenbauer polynomials 
\cite[Sections 18.2-3]{NIST:DLMF}
\begin{equation}
\int_{-1}^1 C_n^\mu(x) C_{n'}^\mu(x) (1-x^2)^{\mu-\frac12}
\,{\mathrm d}x=\dfrac{\pi\Gamma(2\mu+n)}{2^{2\mu
-1}(\mu+n)n!\Gamma^2(\mu)}\, \delta_{n,n'},
\label{orthogGeg}
\end{equation}
where $n,n'\in \mathbb N_0$, 
$\mu\in(-\frac12,\infty)\setminus\{0\}$,
with $\nu=k+\mu$ (resp.~$\nu=k'+\mu$)
%, and $\mu=l+\alpha$ 
produces \eqref{orthogonalityPmk}.
The restriction on 
%$\Re(l+\alpha)$ 
$\Re\mu$ comes from ensuring that the singularities 
at $x=\pm 1$ are integrable. This completes the proof.
\end{proof}

\noindent 
Now specializing the above orthogonality relation using %$\alpha$ 
\boro{$\mu$} as either an integer or an odd-half-in\-te\-ger 
produces \boro{other} orthogonality relations as corollaries.

\boro{
\begin{cor} \label{corPPpos}
%\noindent 
%Let $n,k,k',l\in \mathbb N_0$, $2n+k+l\ge 0$, $n+k+\frac12\ne 0$, 
Let $n,k,k'\in \mathbb N_0$.
%$l+n\ge 0$. 
Then
\begin{equation}
%\int_{-1}^{1}{\sf P}_{k+n}^{l+n}(x){\sf P}_{k'+n}^{l+n}(x) **
\int_{-1}^{1}{\sf P}_{k+n}^{n}(x){\sf P}_{k'+n}^{n}(x)
\,{\mathrm d}x
%=\dfrac{(k+1)_{2n}}{(n+k+\frac12)}
%=\frac{(2n+k+l)!}{(k-l)!(n+k+\frac12)}\delta_{k,k'}**
=\dfrac{(2n+k)!}{k!(n+k+\frac12)}
\,\delta_{k,k'}.
\label{orthPtwo}
\end{equation}
\end{cor}}
\begin{proof}
Let $\mu=n\in\mathbb N_0$ in \eqref{orthogonalityPmk}, and 
the connection relation \cite[(14.9.3)]{NIST:DLMF} expresses 
the Ferrers functions of the first kind with negative integer 
order in terms of Ferrers functions of the first kind with 
positive integer order.
%\hl{The restriction on $l+n$ ensures that the singularities at 
%$x=\pm 1$ are integrable.}
\end{proof}

\boro{
Similarly we can derive orthogonality with non-positive
integer order.}
\boro{
\begin{cor}
%\noindent Let $n,k,k',l\in \mathbb N_0$, $2n+k+l\ge 0$,
%\moro{[RCS: This is a particular case, it is unnecessary, I believe.]}
Let $n,k,k'\in \mathbb N_0$.
%$n+k+\frac12\ne 0$, $l+n\ge 0$. 
Then
\begin{equation}
%\int_{-1}^1 {\sf P}_{k+n}^{-l-n}(x)
\int_{-1}^1 {\sf P}_{k+n}^{-n}(x)
%{\sf P}_{k'+n}^{-l-n}(x){\mathrm d}x=
{\sf P}_{k'+n}^{-n}(x)\,{\mathrm d}x=
%\dfrac{1}{(k+1)_{2n}(n+k+\frac12)}
\dfrac{k!}{(2n+k)!(n+k+\frac12)}
\,\delta_{k,k'}.
\label{orthnegm}
\end{equation}
\end{cor}
}
\boro{
\begin{proof}
Setting $\mu=n\in\mathbb N_0$ in Theorem \ref{mainorthogthm} 
the result follows.
%Using the connection relation \cite[(14.9.3)]{NIST:DLMF} twice
%with \eqref{orthPtwo} completes the proof.
\end{proof}
}

\boro{
\begin{cor}
%\noindent Let $n,k,k',l\in \mathbb N_0$, $2n+k+l\ge 0$, 
%$n+k+\frac12\ne 0$, $l+n\ge 0$. 
Let $n,k,k'\in \mathbb N_0$. 
Then
\begin{equation} \label{orthnegm1}
%\int_{-1}^1 {\sf P}_{k+n}^{l+n}(x)
\int_{-1}^1 {\sf P}_{k+n}^{n}(x)
%{\sf P}_{k'+n}^{-l-n}(x){\mathrm d}x=
{\sf P}_{k'+n}^{-n}(x)\,{\mathrm d}x=
%\frac{(-1)^{l+n}}{n+k+\frac12}
\frac{(-1)^{n}}{n+k+\frac12}\,
\delta_{k,k'}.
\end{equation}
\end{cor}
}
\boro{
\begin{proof}
Using the connection relation \cite[(14.9.3)]{NIST:DLMF} once 
with \eqref{orthPtwo} completes the proof.
\end{proof}
}

\boro{
\begin{rem}
The orthogonality relations \eqref{orthogonalityPmk}, 
\eqref{orthPtwo} generalize \cite[(14.17.6)]{NIST:DLMF},
\begin{equation}
\int_{-1}^{1}{\sf P}_{k}^{l}(x){\sf P}_{k'}^{l}(x)\,{\mathrm d}x
=\frac{(k+l)!}{(k-l)!(k+\frac12)}\delta_{k,k'},
\end{equation}
\boro{by setting $n\mapsto n+l$, and taking} 
$n=0$ in \eqref{orthPtwo}.
Similarly, the orthogonality relation \eqref{orthnegm1} 
generalizes \cite[(14.17.7)]{NIST:DLMF},
\begin{equation}
\int_{-1}^{1}{\sf P}_{k}^{l}(x) 
{\sf P}_{k'}^{-l}(x)\,{\mathrm d}x
=\frac{(-1)^{l}}{k+\frac12}\delta_{k,k'},
\end{equation} 
by setting $n=0$.
\end{rem}
}

\noindent 
One also has the following orthogonality relation for the Ferrers 
function of the second kind.

\boro{\begin{cor} \label{corQQpos}
%\noindent Let $n,k,l\in\mathbb Z$, $2n+k+l+1\ge 0$, $n+k+1\ne 0$, 
Let $n, k, k'\in\mathbb N_0$.
%$l+n\ge -1$. 
%$n\ge -1$.
Then
\begin{equation} \label{QQorthog}
%\int_{-1}^{1}{\sf Q}_{k+n+\frac12}^{l+n+\frac12}(x)
%{\sf Q}_{k'+n+\frac12}^{l+n+\frac12}(x)
\int_{-1}^{1}{\sf Q}_{k+n-\frac12}^{n-\frac12}(x)
{\sf Q}_{k'+n-\frac12}^{n-\frac12}(x)
\,{\mathrm d}x
%=\frac{\pi^2(2n+k+l+1)!}{4(k-l)!(n+k+1)}\delta_{k,k'}.
=\dfrac{\pi^2(2n+k-1)!}{4(n+k)k!}\,\delta_{k,k'}.
\end{equation}
\end{cor}}
\begin{proof}
Let $\mu=n-\frac12$, $n\in\mathbb N_0$ in \eqref{orthogonalityPmk}, 
and the connection relation \cite[(14.9.1)]{NIST:DLMF} 
expresses the Ferrers functions of the first kind with
negative integer order in terms of Ferrers functions of 
the second kind with positive integer order.
%\hl{Again, the restriction on $l+n$ ensures that the 
%singularities at $x=\pm 1$ are integrable.}
\end{proof}

\begin{cor}
\noindent 
Let $l,k,k'\in\mathbb Z$, $l+k+1\ge 0$, $k+1\ne 0$, $l\ge -1$. 
Then
\begin{equation}
\int_{-1}^{1}{\sf Q}_{k+\frac12}^{l+\frac12}(x)
{\sf Q}_{k'+\frac12}^{l+\frac12}(x)\,{\mathrm d}x
=\frac{\pi^2(k+l+1)!}{4(k-l)!(k+1)}\delta_{k,k'}.
\label{QQorthogtwo}
\end{equation}
\end{cor}
\begin{proof} 
Specializing \eqref{QQorthog} with $n=0$ completes the proof.
\end{proof}

\boro{
\begin{rem}
For $l\in\{-1,0\}$, \eqref{QQorthogtwo} reduces to 
orthogonality for trigonometric functions, or
Chebyshev polynomials of the first and second kinds.
For $l=-1$, \eqref{QQorthogtwo} reduces to
orthogonality for Chebyshev polynomials
of the first kind
\eqref{Cheby},
namely \cite[(9.8.37)]{Koekoeketal}
\begin{equation}
\int_{0}^\pi \cos(k\theta)\cos(k'\theta)\,{\mathrm d}\theta=
\int_{-1}^1T_k(x)T_{k'}(x)\frac{{\mathrm d}x}{\sqrt{1-x^2}}=
\frac{\pi}{\epsilon_k}\delta_{k,k'},
\end{equation}
where $\epsilon_k=2-\delta_{k,0}$ is the 
Neumann factor.
For $l=0$, \eqref{QQorthogtwo} reduces to 
\begin{equation}
\int_{0}^\pi\sin(k\theta)\sin(k'\theta)\,{\mathrm d}\theta
=\delta_{k,k'}\begin{cases}
0, & {\mathrm{if}}\, k=0, \\
\frac{\pi}{2}, & {\mathrm{if}}\, k\ge 1,
\end{cases}
\end{equation}
or for $k\ge 0$, \cite[(9.8.38)]{Koekoeketal}
\begin{equation}
\int_{-1}^1\sin((k+1)\theta)\sin((k'+1)\theta)\,{\mathrm d}
(\cos\theta)=\int_{-1}^1 U_k(x)U_{k'}(x)
\sqrt{1-x^2}\, {\mathrm d}x=\frac{\pi}{2}\,\delta_{k,k'}.
\end{equation}
\label{trigorth}
\end{rem}
}
%%%%%%%%%%%%%%%%%%%%%%%%%%%%%%%%%%%%%%%%%%%%%%%%%%%%%%%%%%%%%%%%%%%%%
\subsection{\boro{Sobolev orthogonality
from Gegenbauer and Meixner polynomials}}
%%%%%%%%%%%%%%%%%%%%%%%%%%%%%%%%%%%%%%%%%%%%%%%%%%%%%%%%%%%%%%%%%%%%%
\begin{thm}
Let $N, k, k'\in \mathbb N_0$, \poro{$x\in(-1,1)$}. Then the \poro{Ferrers function of the first kind}
fulfills the property of orthogonality 
\begin{equation}
{\mathcal B}\left({\sf P}_{k}^{N}(x),{\sf P}_{k'}^{N}(x)
\right)=\delta_{k,k'} \begin{cases}h^I_k,\ {\mathrm{if}}\ 
k< 2N,\\[2mm] h^{II}_k,\ {\mathrm{if}}\ k\ge2N,\end{cases}
\end{equation}
where 
\begin{eqnarray}
h^I_k=
\frac{
(\frac12\!-\!N,1\!-\!2N)_k}{
k!(\frac32\!-\!N)_k}
=\frac{(-1)^k(2N\!-\!1) (2N\!-\!1)!}
{k!(2N\!-\!2k\!-\!1)(2N\!-\!k\!-\!1)!},\\
%\hspace{-1cm}
h_{k}^{II}=
\frac{2 k!(2N\!-\!1)!!^2}
{(1\!+\!2k\!-\!2N) (k\!-\!2N)!}
=\frac{2^{2N+1}k!\left(\frac12\right)^2_N}
{(1\!+\!2k\!-\!2N)(k\!-\!2N)!},
\end{eqnarray}
${\mathcal B}$ is the bilinear form defined as follows 
\cite[Theorem 4.3]{salara}
\begin{equation}
{\mathcal B}(f,g):=\left((1-x^2)^{\frac N2}f,(1-x^2)^{\frac N2}
g\right)_1+\int_{-1}^1 ((1-x^2)^{\frac N2}f)^{(2N)}(x)((1-
x^2)^{\frac N2}g)^{(2N)}(x)(1-x^2)^N \,{\mathrm d}x,
\end{equation}
where 
\begin{equation}
(f,g)_1:=\sum_{k=0}^{N-1} \binom{N-1}{k}
\frac{2^{k-1}}{(2-2N)_k}
\Biggl[\left(\frac{{\mathrm d}^k}{{\mathrm d}x^k} 
f(x)g(x)\right)_{x=1}+(-1)^k\left(\frac{{\mathrm d}^k}
{{\mathrm d}x^k} f(x)g(x)\right)_{x=-1}\Biggr].
\end{equation}
\end{thm}

\begin{proof}
The proof follows by using the identity \eqref{FerPnlammlan},
the recurrence relations for the Gegenbauer polynomials 
\cite[(9.8.21)]{Koekoeketal}, and the 
property of orthogonality of the Gegenbauer polynomials.
If $k<2N$, then 
\[\begin{split}
h^I_k&=\!\!\left({\sf P}_{k}^{N}(x)(1-x^2)^{\frac N2},
{\sf P}_{k}^{N}(x)(1-x^2)^{\frac N2}\right)_1%\\ 
%&
\!\!=\!\left(\frac{\sqrt{\pi}k!}{2^{N} 
\Gamma(N+\tfrac12)(k-2N)!}\right)^2\!
\left(C_k^{\frac12-N}(x),C_k^{\frac12-N}(x)\right)_1\\
&=\!\!\left(\frac{\sqrt{\pi}k!}{2^{N} 
\Gamma(N+\tfrac12)(k-2N)!}\right)^2\!\!
\left(C_k^{\frac12-N}(x), \frac{x(2k-1-2N)}{k}C_{k-1}^{\frac12-N}(x)
\!-\dfrac{k-1-2N}kC_{k-2}^{\frac12-N}(x)\right)_1\\
&=\frac{(2k-1-2N)}{k}\left(\frac{\sqrt{\pi}k!}{2^{N}
\Gamma(N+\tfrac12)(k-2N)!}\right)^2\left(x C_k^{\frac12-N}(x), 
C_{k-1}^{\frac12-N}(x)\right)_1\\
%&=\frac{(2k-1-2N)}{k}\left(\frac{\sqrt{\pi}k!}{2^{N}
%\Gamma(N+\tfrac12)(k-2N)!}
%\right)^2
%\left(\frac k{2k+1-2N}C_{k+1}^{\frac12-N}(x)+
%\dfrac{k-2N}{2k+1-2N}C_{k-1}^{\frac12-N}(x), 
%C_{k-1}^{\frac12-N}(x)\right)_1\\
&=\frac{(2k-1-2N)}{k}\dfrac{k-2N}{2k+1-2N}
\left(\frac{\sqrt{\pi}k!}{2^{N}\Gamma(N+\tfrac12)(k-2N)!}\right)^2
\left(C_{k-1}^{\frac12-N}(x), C_{k-1}^{\frac12-N}(x)\right)_1.
\end{split}\]
After some straightforward calculations we get the 
desired value.
If $k\ge 2N$, due to the following factorization identity:
\[
C_{k}^{\frac12-N}(x)=(x^2-1)^N C_{k-2N}^{\frac12+N}(x),
\]
using \eqref{relFerPGeg} and after some algebraic manipulations 
one demonstrates
\[
\left({\sf P}_{k}^{N}(x)(1-x^2)^{\frac N2},
{\sf P}_{k}^{N}(x)(1-x^2)^{\frac N2}\right)_1=0.
\]
Moreover, taking into account the identity
\cite[Remark 12]{CohlCostasSantos2020}
\begin{equation}
\frac{{\mathrm d}^n}{{\mathrm d}x^n}
\frac{{\sf P}_\nu^{-\mu}(x)}{(1-x^2)^{\frac{\mu}{2}}}
=(-1)^n (\nu+\mu+1)_n(\mu-\nu)_n
\frac{{\sf P}_\nu^{-\mu-n}(x)}{(1-x^2)^{\frac{\mu+n}{2}}},
\label{derivnFerPnummu}
\end{equation}
then 
\[\begin{split}
h^{II}_k=&
\int_{-1}^1 ((1-x^2)^{\frac N2}{\sf P}_{k}^{N}(x))^{(2N)}(x)
((1-x^2)^{\frac N2}{\sf P}_{k}^{N}(x))^{(2N)}(x)
(1-x^2)^N \,{\mathrm d}x\\
=&(k-N+1)^2_{2N} (-N-k)^2_{2N}\int_{-1}^1 p_k^{-N}(x)
p_k^{-N}(x) (1-x^2)^{-N}\, {\mathrm d}x\\
=& \frac {(k-N+1)^2_{2N} (-N-k)^2_{2N}}{(2N+1)^2_{k-N}}
\int_{-1}^1 C_{k-N}^{N+1/2}(x)C_{k-N}^{N+1/2}(x)\, {\mathrm d}x.
\end{split}\]
After some straightforward calculations we obtain the desired 
value. Hence the result follows.
\end{proof}

More details about the  Gegenbauer/ultraspherical polynomials
for non-classical parameters, their recurrence relations, 
and the Sobolev-type orthogonality can be found in \cite{salara}.

\section{\boro{Properties 
which follow from orthogonality \boro{and completeness}}}
%%%%%%%%%%%%%%%%%%%%%%%%%%%%%%%%%%%%%%%%%%%%%%%%%%%%%%%%%%%%%%%%%%%%%
%\noindent \goro{\bf [HSC: Are there any other orthogonality 
%properties we can extract? Look at \cite{Koekoeketal}.]
%}
%%%%%%%%%%%%%%%%%%%%%%%%%%%%%%%%%%%%%%%%%%%%%%%%%%%%%%%%%%%%%%%%%%%%%
\subsection{\boro{Connection and linearization relations}}
%%%%%%%%%%%%%%%%%%%%%%%%%%%%%%%%%%%%%%%%%%%%%%%%%%%%%%%%%%%%%%%%%%%%%
\boro{
Here we find the connection and linearization
relations which arise from Gegenbauer 
polynomials for the Ferrers function of the
first kind.}

\boro{
\begin{cor}
Let $n\in \mathbb N_0$, $\nu,\mu\in\mathbb C$. Then
\begin{eqnarray}
&&\hspace{-0.9cm}{\sf P}_{n+\nu}^{-\nu}(x)
=\frac{2^{\nu-\mu}n!(1-x^2)^{\frac{\nu-\mu}{2}}}
{\Gamma(2\nu+2+n)}\nonumber\\
&&\times\sum_{k=0}^{\lfloor\frac{n}{2}\rfloor}
(\mu+n-2k+\tfrac12)\frac{(\nu-\mu)_k\Gamma(\nu+\frac12+n-k)
\Gamma(2\mu+n-2k+1)}{k!(n-2k)!\Gamma(\mu+\frac32+n-k)}
{\sf P}_{n-2k+\mu}^{-\mu}(x).
\end{eqnarray}
\end{cor}
}
\boro{
\begin{proof}
The connection relation for Gegenbauer polynomials
is given by
\cite[(18.18.16)]{NIST:DLMF}
\begin{equation}
C_n^\nu(x)=\frac{1}{\mu}
\sum_{k=0}^{\lfloor\frac{n}{2}\rfloor}
(\mu+n-2k)\frac{(\nu-\mu)_k(\nu)_{n-k}}
{k!(\mu+1)_{n-k}}
C_{n-2k}^{\mu}(x).
\label{con:gegen}
\end{equation}
First replace the Gegenbauer polynomials in 
\eqref{con:gegen} using
\eqref{CnmuFerP} then 
setting 
$\lambda\mapsto\lambda+\frac12$ completes the proof.
\end{proof}
}

\boro{
\begin{cor}
Let $n,m\in \mathbb N_0$, $\lambda\in\mathbb C$. Then 
\begin{eqnarray}
&&\hspace{-0.5cm}{\sf P}_{n+\lambda}^{-\lambda}(x)
{\sf P}_{m+\lambda}^{-\lambda}(x)\nonumber\\
&&=\frac{2^{\lambda}\Gamma(\lambda+\frac12)m!n!
(1-x^2)^{\frac{\lambda}{2}}}
{\sqrt{\pi}\,\Gamma(2\lambda+m+1)\Gamma(2\lambda+n+1)}
\sum^{m}_{k=0} B^{\lambda+\frac12}_{k,m,n}
\frac{\Gamma(2\lambda+n+m-2k+1)}
{(n+m-2k)!}
{\sf P}_{n+m-2k+\lambda}^{-\lambda}(x).
\label{linFerP}
\end{eqnarray}
\end{cor}
}

\boro{
\begin{proof}
The Gegenbauer polynomials 
have a linearization formula given by \cite[(18.18.22)]{NIST:DLMF}
\begin{equation}
C^\lambda_m (x) C^\lambda_n (x) = \sum^{m}_{k=0} B^{\lambda}_{k,m,n}
C^{\lambda}_{m+n-2k} (x),
\label{lin:gegen}
\end{equation}
where $m,n \in \mathbb{N}_0$, $n\geq m$, and
\begin{equation*}
B^{\lambda}_{k,m,n} := \frac{(m+n+ \lambda - 2k)(m+n-2k)!(\lambda)_k
(\lambda)_{m-k} (\lambda)_{n-k} (2\lambda)_{m+n-k}}{(m+n+\lambda -k)
k! (m-k)!(n-k)! (\lambda)_{m+n-k} (2\lambda)_{m+n-2k}}.
\end{equation*}
First replace the Gegenbauer polynomials in 
\eqref{lin:gegen} using
\eqref{CnmuFerP} then 
setting 
$\lambda\mapsto\lambda+\frac12$ completes the proof.
\end{proof}
}
%%%%%%%%%%%%%%%%%%%%%%%%%%%%%%%%%%%%%%%%%%%%%%%%%%%%%%%%%%%%%%%%%%%%%
\subsection{\boro{Some definite integrals and sums}}
%%%%%%%%%%%%%%%%%%%%%%%%%%%%%%%%%%%%%%%%%%%%%%%%%%%%%%%%%%%%%%%%%%%%%
\boro{
Along the lines of 
\cite{CohlVolkmerDefInt}, one can use these
orthogonality relations to compute some new definite integrals.
Furthermore, using the derived closure relations one can convert 
certain definite integrals into infinite series expressions.
In this subsection, we give some carefully chosen 
examples of these procedures.
}

\medskip
\noindent \boro{
The following result follows from a sample generating
function for the Gegenbauer polynomials
\eqref{genfunGeg} and then re-expressing
as a definite integral using orthogonality
for the Ferrers function of the first kind.
}

\boro{
\begin{cor}
Let $|x|,|t|<1$, $\alpha,\gamma\in{\mathbb C}$, $k,\ell\in
{\mathbb N}_0$, $\alpha+k+\frac12\ne 0$, $\Re\alpha>-\ell-1$. 
Then
\begin{eqnarray}
&&\hspace{-1cm}\int_{-1}^1 (1+t^2-2xt)^{\frac{\alpha-\gamma+\ell}{2}}
{\sf P}_{\ell+\alpha-\gamma}^{-\ell-\alpha}\left(
\frac{1-xt}{\sqrt{1+t^2-2xt}}\right)
{\sf P}_{k+\alpha}^{-\ell-\alpha}(x)\,{\mathrm d}x
\nonumber\\
&&\hspace{5cm}=\frac{(-1)^\ell(\gamma-\ell)_k 
t^{\alpha+k}}{(1-\gamma)_\ell\,
\Gamma(2\alpha+\ell+k+1)(\alpha+k+\frac12)}.
\end{eqnarray}
Special care must be taken when $\gamma\in{\mathbb Z}$.
\end{cor}
}

\boro{
\begin{proof}
Start with the generating function $\lambda,\gamma\in{\mathbb C}$,
\cite[(9.8.33)]{NIST:DLMF}
\begin{equation}
(1-xt)^{-\gamma}
\hyp21{\frac{\gamma}{2},\frac{\gamma+1}{2}}{\lambda+\frac12}
{\frac{t^2(x^2-1)}{(1-xt)^2}}=\sum_{n=0}^\infty
\frac{(\gamma)_n}{(2\lambda)_n}
C_n^\lambda(x) t^n,
\label{genfunGeg}
\end{equation}
and making the replacement $n\mapsto k-\ell$, 
$k\ge \ell\ge 0$ converts the
sum over all $k\ge \ell\in{\mathbb N}_0$.
Then 
convert the Gegenbauer polynomial to a Ferrers function
of the first kind using \eqref{relFerPGeg}.
Multiply the resulting sum by $(1-x^2)^{\frac{\alpha
+\ell}{2}}{\sf P}_{k'+\alpha}^{-\ell-\alpha}(x)$ 
and integrate both sides of the equation
over $(-1,1)$ 
using the orthogonality relation 
\eqref{orthogonalityPmk} produces
a definite integral of the Gauss
hypergeometric function multiplied
by a Ferrers function of the first kind. One can then
convert the Gauss hypergeometric
function to a Ferrers function
of the first kind using 
\eqref{FerPnlammlan}
which completes the proof.
\end{proof}
}

%\medskip
\noindent\boro{
The following result follows from an expansion
over Gegenbauer polynomials 
\cite[Corollary 1]{Cohl12pow} which generalizes
many classical expansions including the generating 
function for Legendre polynomials, Heine's formula 
and Heine's reciprocal square root Fourier cosine 
series expansion \cite[Figure 1]{CohlCostasWakhare2019}.
}

\boro{
\begin{cor}
Let $\nu,\mu,x\in\mathbb C$ such that if $z\in\mathbb 
C\setminus(-\infty,1]$ lies on any ellipse with foci at 
$\pm 1$ then $x$ can lie on any point interior to that 
ellipse. Then
\begin{eqnarray}
&&\hspace{-0.7cm}\frac{ (1\!-\!x^2)^{\frac{\mu}{2}} 
(z^2\!-\!1)^{\frac{\nu-\mu-1}{2}}}{(z\!-\!x)^\nu}
%\nonumber\\&&\hspace{5cm}
=\sum_{n=0}^\infty
\frac{(2\mu+2n+1)(\nu)_n\Gamma(2\mu+n+1)}{n!}
{\sf P}_{n+\mu}^{-\mu}(x)\biQ_{n+\mu}^{\nu-\mu-1}(z).
\label{sumzmxnuFer}
\end{eqnarray}
\end{cor}
}
\boro{
\begin{proof}
Start with 
\cite[Corollary 1]{Cohl12pow}
\[
\frac{1}{(z-x)^\nu}=\frac{2^{\mu+\frac12}\Gamma(\mu)}
{\sqrt{\pi}\,(z^2-1)^{\frac{\nu-\mu}{2}-\frac14}}
\sum_{n=0}^\infty (n+\mu)(\nu)_n
\biQ_{n+\mu-\frac12}^{\nu-\mu-\frac12}(z)
C_n^\mu(x),
\]
and convert the Gegenbauer polynomial to a 
Ferrers function of the first kind using
\eqref{CnmuFerP}, then set $\mu\mapsto\mu+\frac12$.
This completes the proof.
\end{proof}
}

\boro{
\begin{cor} 
Let $n,l\in \mathbb N_0$, $n\ge l$, $\nu,\mu\in{\mathbb C}$, 
such that $z\in{\mathbb C}\setminus(-\infty,1]$. Then
\begin{eqnarray}
&&\hspace{-1cm}\int_{-1}^1 
\frac{{\sf P}_{n+\mu}^{-\mu}(x)}
{(z-x)^\nu}(1-x^2)^\frac{\mu}{2}\,{\mathrm d}x=
\frac{
2(\nu)_n\biQ_{n+\mu}^{\nu-\mu-1}(z)
}{(z^2-1)^\frac{\nu-\mu-1}{2}}.
\label{intFerPzmxnua}
\end{eqnarray}
%or in terms of orthogonal Ferrers functions
%of the first kind
%\begin{eqnarray}
%&&\hspace{-1cm}\int_{-1}^1 
%\frac{{\sf P}_{n+\mu}^{-\mu-l}(x)}
%{(z-x)^\nu}(1-x^2)^\frac{\mu+l}{2}{\mathrm %d}x=\frac{
%%2e^{-i\pi(\nu-\mu-l-1)}}{\Gamma(\nu)(z^2-1)^\frac
%{\nu-\mu-l-1}{2}}Q_{n%+\mu}^{\nu-\mu-l-1}(z).
%\label{intFerPzmxnu}
%\end{eqnarray}
\end{cor}
}
\boro{
\begin{proof}
Starting with \eqref{sumzmxnuFer}, multiplying both sides 
of the equation by 
${\sf P}_{n+\mu}^{-\mu}(x)$ and integrating over $(-1,1)$ using the 
orthogonality relation \eqref{orthogonalityPmk} completes the proof.
\end{proof}
}

\medskip
\noindent \boro{
The following result follows from an integral
derived in \cite{AskeyRazban72} by Askey and Razban %(1972) 
for the Jacobi polynomials, but then we specialize to
Gegenbauer polynomials and then re-express using
the Ferrers function of the first kind.
}

\boro{
\begin{cor}
Let $n\in \mathbb N_0$, $\lambda,\gamma\in\mathbb C$. Then
\begin{equation}
\int_{-1}^1(1-x)^{\frac{\lambda}{2}-\gamma}(1+x)^{\frac{\lambda}{2}}
{\sf P}_{n+\lambda}^{-\lambda}(x)
\,{\mathrm d}x
=
\frac{2^{\lambda-\gamma+1}\Gamma(\lambda-\gamma+1)\Gamma(\gamma+n)}
{\Gamma(\gamma)\Gamma(2\lambda-\gamma+n+2)}.
\label{thisintFerP}
\end{equation}
\end{cor}
}

\boro{
\begin{proof}
Start with \cite{AskeyRazban72} 
(see also \cite[(16.4.2)]{ErdelyiTII})
\[
\int_{-1}^{1}P_n^{(\alpha,\beta)}(x)\
(1-x)^{\alpha-\gamma}(1+x)^{\beta}\,{\mathrm d}x
=\frac{2^{\alpha+\beta-\gamma+1}\Gamma(\alpha-\gamma+1)
\Gamma(\beta+n+1)\Gamma(\gamma+n)}
{n!\Gamma(\gamma)\Gamma(\alpha+\beta-\gamma+n+2)},
\]
then replace $\alpha,\beta\mapsto\lambda-\frac12$
and using \eqref{GegJac}, then \eqref{CnmuFerP}
produces the following integral over Gegenbauer
polynomials
\begin{equation}
\int_{-1}^1(1-x)^{\lambda-\gamma-\frac12}(1+x)^{\lambda-\frac12}
C_n^\lambda(x)\,{\mathrm d}x=
\frac{\sqrt{\pi}\Gamma(\lambda-\gamma+\frac12)\Gamma(2\lambda+n)
\Gamma(\gamma+n)}{2^{\gamma-1}n!\Gamma(\lambda)\Gamma(\gamma)
\Gamma(2\lambda-\gamma+n+1)},
\end{equation}
using \eqref{CnmuFerP} and replacing 
$\lambda\mapsto\lambda+\frac12$ completes
the proof.
\end{proof}
}

\boro{
\begin{rem}
The above result is equivalent to
the following expansion
\begin{equation}
\sum_{n=0}^\infty
\frac{(\lambda+n)(\gamma)_n}
{(2\lambda-\gamma+1)_n}
C_n^\lambda(x)
=\frac{
\sqrt{\pi}\Gamma(2\lambda-\gamma+1)(1-x)^{-\gamma}}
{2^{2\lambda-\gamma}\gamma(\lambda)\Gamma(\lambda-\gamma+\frac12)},
\end{equation}
which we recently proved in 
\cite[Corollary 5.12]{CohlCostasWakhare2019}.
\end{rem}
}

\subsection{\boro{Asymptotic expansions
and Mehler-Heine relations}}

\noindent \boro{
The asymptotic expansion of
the Ferrers function of the 
first kind which is
related to the limit transition
\cite[(18.6.4)]{NIST:DLMF}
\begin{equation}
\lim_{\lambda\to\infty}
\frac{C_n^\lambda(x)}
{C_n^\lambda(1)}=n!\lim_{\lambda\to\infty}
\frac{C_n^\lambda(x)}
{(2\lambda)_n}
=x^n,
\label{Gegxn}
\end{equation}
since \cite[Table 18.6.1]{NIST:DLMF}
$
C_n^\lambda(1)=\frac{(2\lambda)_n}{n!},
$
is given as follows.
}

\boro{
\begin{cor}
Let $n\in \mathbb N_0$, $x\in(-1,1)$. Then
\begin{equation}
\hspace{-0.4cm}\frac{2^\lambda\Gamma(\lambda\!+\!1)}
{(1\!-\!x^2)^\frac{\lambda}{2}}
{\sf P}_{n+\lambda}^{-\lambda}(x)
= x^n
\left(1\!-\!
\frac{n(n\!-\!1) (1\!-\!x^2)}
{4x^2\lambda}
\!+\!\frac{n(n\!-\!1)(n\!-\!2)(n\!-\!3)
(1\!-\!x^2)^2}
{32x^4\lambda^2}
\!+\!
{\mathcal O}\left(\frac{1}{\lambda^3
}\right)
\right),
\end{equation}
as $\lambda\to\infty$.
\end{cor}
}

\boro{
\begin{proof}
Start with
\eqref{Gegxn}
then apply \eqref{CnmuFerP}
with the replacement 
$\lambda\mapsto\lambda+\frac12$.
This gives the zeroth order 
term. In order to compute 
an arbitrary asymptotic expansion
near $\lambda=\infty$,
use \eqref{FerPnlammlan},
expanding the resulting
terminating expansion 
using 
a Laurent series 
near $\lambda=\infty$ with
\eqref{DLMFhypasymp}.
This completes the proof.
\end{proof}
}

\medskip
\boro{
\begin{thm}
\label{thm310}
Let $n\in \mathbb N_0$, $x\in(-1,1)$. Then
\begin{eqnarray}
&&\hspace{-0.5cm}\frac{\sqrt{\pi}\,\Gamma(2\lambda+n+1)}
{2^{\lambda}(\lambda+\frac12)^{\frac{n}{2}}
\Gamma(\lambda+\frac12)(1-\frac{x^2}{\lambda})^{\frac{\lambda}{2}}}
{\sf P}_{n+\lambda}^{-\lambda}
\left(x\left(\lambda+\tfrac12\right)^{-\frac12}\right)
\nonumber \\ &&=(2x)^n\hyp21{-\frac{n}{2},\frac{1-n}{2}}
{\lambda+1}{-\frac{\lambda}{x^2}}
\left(1+\frac{n(n+1)}{4\lambda}
+\frac{n(n-1)(n+1)(3n+2)}{96\lambda^2}
+{\mathcal O}\left(\frac{1}{\lambda^3}\right)\right),
\end{eqnarray}
as $\lambda\to\infty$.
\end{thm}
}
\boro{
\begin{proof}
Using \eqref{FerPnlammlan} one 
obtains
\begin{eqnarray}
&&\hspace{-1cm}\frac{\sqrt{\pi}\,\Gamma(2\lambda+n+1)}{2^{\lambda}
(\lambda+\frac12)^{\frac{n}{2}}\Gamma(\lambda+\frac12)(1-\frac{x^2}
{\lambda})^{\frac{\lambda}{2}}}{\sf P}_{n+\lambda}^{-\lambda}
\left(x
\left(\lambda+\tfrac12\right)^{-\frac12}\right)
\nonumber\\
&&\hspace{4cm}=\frac{\Gamma(2\lambda+n+1)x^n}{\Gamma(2\lambda)
(\lambda+\tfrac12)^n}\hyp21{-\frac{n}{2},\frac{1-n}{2}}
{\lambda+1}{1-\frac{\lambda+\frac12}{x^2}}.
\end{eqnarray}
Then by performing an asymptotic expansion up to second order
in $1/\lambda$ using \cite[(5.11.12)]{NIST:DLMF} and expanding 
the resulting terminating expansion using a Laurent series 
near $\lambda=\infty$, produces the result.
\end{proof}
}

\boro{
\begin{cor} \label{cor311}
Let $n\in \mathbb N_0$, $x\in(-1,1)$. Then
\begin{equation}
\frac{2^{\lambda}\Gamma(\lambda+1)(2\lambda+1)_n}
{(\lambda+\tfrac12)^{\frac{n}{2}}(1-x^2)^{\frac{\lambda}{2}}}
{\sf P}_{n+\lambda}^{-\lambda}
\left(x
\left(\lambda+\tfrac12\right)^{-\frac12}\right)
\sim
H_n(x),
\end{equation}
as $\lambda\to\infty$,
where $H_n(x)$ are the Hermite polynomials.
\end{cor}
}

\begin{proof}
Start with \cite[(18.7.24)]{NIST:DLMF}
\[
\lim_{\lambda\to\infty}\frac{1}{\lambda^{\frac{n}{2}}}
C_{n}^\lambda\left(\frac{x}{\sqrt{\lambda}}\right)=
\frac{1}{n!}
H_n(x),
\]
then inserting \eqref{CnmuFerP} produces
\[
\lim_{\lambda\to\infty}
\frac{\sqrt{\pi}\,\Gamma(2\lambda+n)}
{2^{\lambda-\tfrac12}\lambda^{\frac{n}{2}}\Gamma(\lambda)
(1-\frac{x^2}{\lambda})^{\frac{\lambda}{2}-\frac14}}
{\sf P}_{n+\lambda-\frac12}^{\frac12-\lambda}
\left(\frac{x}{\sqrt{\lambda}}\right)=
H_n(x).
\]
Setting $\lambda\mapsto\lambda+\frac12$ produces 
\[
\frac{\sqrt{\pi}\,\Gamma(2\lambda+n+1)}
{2^\lambda(\lambda+\tfrac12)^{\frac{n}{2}}\Gamma(\lambda+\tfrac12)
\left(1-\frac{x^2}{\lambda+\tfrac12}\right)^{\frac{\lambda}{2}}}
{\sf P}_{n+\lambda}^{-\lambda}
\left(x
\left(\lambda+\tfrac12\right)^{-\frac12}\right)
\sim
H_n(x),
\]
which completes the proof.
\end{proof}

\boro{
\begin{rem}
Corollary \ref{cor311} clearly matches up to Theorem 
\ref{thm310} by using \eqref{limgenhyp} and \eqref{Hermite}.
\end{rem}
}

%\medskip
\boro{
Now we examine this relevant
asymptotic expansion in connection with the
Mehler-Heine formula for Gegenbauer polynomials
in terms of the Ferrers function of the first kind.
}

\boro{
\begin{thm}
Let $z,\lambda\in\mathbb C$, $1\ll n\in \mathbb N_0$. Then
\begin{equation}
n^{\lambda}\,
{\sf P}_{n+\lambda}^{-\lambda}\left(1-\frac{z^2}
{2n^2}\right)=J_{\lambda}(z)\left(1+\frac{2\lambda+1}{n}
+{\mathcal O}\left(\frac{1}{n^2}\right)\right),
\end{equation}
as $n\to\infty$.
\end{thm}
}

\boro{
\begin{proof}
The Mehler-Heine formula for Jacobi
polynomials \cite[(18.11.5)]{NIST:DLMF}
\begin{equation}
\lim_{n\to\infty}
\frac{1}{n^\alpha}
P_n^{(\alpha,\beta)}
\left(1-\frac{z}{2n^2}\right)
=\left(\frac{2}{z}\right)^\alpha J_\alpha(z),
\end{equation}
can be used to obtain the Mehler-Heine
formula for Gegenbauer polynomials using
\eqref{GegJac}, namely
\begin{equation}
\label{MHGeg}
\lim_{n\to\infty}
\frac{n!}{\Gamma(2\lambda+n+1)}
C_n^{\lambda+\frac12}\left(1-\frac{z^2}{2n^2}\right)
=\frac{\sqrt{\pi}J_{\lambda}(z)}{\Gamma(\lambda
+\frac12)(2z)^{\lambda}}.
\end{equation}
Now evaluate the left-hand side of \eqref{MHGeg} 
in terms of the Ferrers function of the
first kind with \eqref{CnmuFerP}, which produces
\begin{equation}
C_n^{\lambda+\frac12}\left(1-\frac{z^2}{2n^2}\right)
=\frac{\sqrt{\pi}\,\Gamma(2\lambda+n+1)}
{2^{\lambda}\Gamma(\lambda+\frac12)n!}
\left(1-\left(1-\frac{z^2}{2n^2}\right)^2\right)^{-\frac{\lambda}{2}}
{\sf P}_{n+\lambda}^{-\lambda}\left(1-\frac{z^2}{2n^2}\right),
\end{equation}
then simplification 
using 
\cite[(5.11.12)]{NIST:DLMF}
\[
\frac{\Gamma(z+a)}{\Gamma(z+b)}\sim z^{a-b},
\]
as $z\to\infty$, obtains the prefactor.
Then use \eqref{FerPdefn} and expand 
the numerator Pochhammer symbols 
in the ${}_2F_1$ in descending powers of $n$, 
namely
\[
(-n-\lambda)_k(n+\lambda+1)_k
=(-1)^kn^{2k}\left(1+\frac{2\lambda+1}{n}
+{\mathcal O}\left(\frac{1}{n^2}\right)\right),
\]
and using the definition of the Bessel function
of the first kind \eqref{Besseldef},
completes the proof.
\end{proof}
}

\subsection{\boro{Christoffel-Darboux formulas}}

\boro{
The Christoffel-Darboux formula for orthogonal Ferrers functions 
is given in the following theorem which gives the
Christoffel-Darboux formula and its confluent form for the Ferrers 
function of the first kind.
}

\boro{
\begin{thm}
Let $x,x'\in(-1,1)$, $\mu\in\mathbb C$, $n\in \mathbb N_0$. Then, 
\boro{the Christoffel-Darboux 
formula and its confluent 
form for the Ferrers function of the 
first kind are given by}
\begin{eqnarray}
&&\hspace{0.0cm}\sum_{k=0}^{n-1}
\frac{(2\mu+2k+1)\Gamma(2\mu+k+1)}
{k!}
{\sf P}_{k+\mu}^{-\mu}(x)
{\sf P}_{k+\mu}^{-\mu}(x')\nonumber
\\
&&\hspace{2cm}=
\frac{\Gamma(2\mu+n+1)}
{(n-1)!(x-x')}
\left(
{\sf P}_{n+\mu}^{-\mu}(x)
{\sf P}_{n+\mu-1}^{-\mu}(x')
-
{\sf P}_{n+\mu-1}^{-\mu}(x)
{\sf P}_{n+\mu}^{-\mu}(x')
\right),\label{CD1}\\
&&\hspace{-0.0cm}\sum_{k=0}^{n-1}
\frac{(2\mu+2k+1)\Gamma(2\mu+k+1)}
{k!}
\left({\sf P}_{k+\mu}^{-\mu}(x)
\right)^2
\nonumber\\
&&\hspace{2cm}=
\frac{\Gamma(2\mu+n+1)}
{(n-1)!(1-x^2)}
\biggl(
(n\!+\!2\mu)
\left({\sf P}_{n+\mu}^{-\mu}(x)
\right)^2
+n\!\left(
{\sf P}_{n+\mu-1}^{-\mu}(x)
\right)^2
\nonumber\\
&&\hspace{8.5cm}
-2x(n\!+\!\mu)
{\sf P}_{n+\mu-1}^{-\mu}(x)
{\sf P}_{n+\mu}^{-\mu}(x)
\biggr).
\label{CD2}
\end{eqnarray}
\end{thm}
}

\boro{
\begin{proof}
Start with the Christoffel-Darboux
summation formula for the
Gegenbauer polynomials 
\cite[(3.10)]{Szmy2}
\begin{equation}
\sum_{k=0}^{n-1}
\frac{k!(k+\lambda)}{\Gamma(2\lambda+k)}
C_k^\lambda(x)
C_k^\lambda(x')
=\frac{n!\left(
C_n^\lambda(x)C_{n-1}^\lambda(x')
-C_{n-1}^\lambda(x)C_n^\lambda(x')\right)}
{2\Gamma(2\lambda-1+n)(x-x')}.
\end{equation}
The confluent form of the Christoffel-Darboux
summation formula follows from the first formula
using \cite[(18.2.13)]{NIST:DLMF} 
and recurrence relations for the Gegenbauer polynomials 
\cite[(9.8.21)]{Koekoeketal},
\cite[(8.933.4)]{Grad},
and is given by
\begin{eqnarray}
&&\hspace{-1.5cm}\sum_{k=0}^{n-1}\frac{k!(k+\lambda)}
{\Gamma(2\lambda\!+\!k)}\left(C_k^\lambda(x)\right)^2\nonumber\\
&&\hspace{0.5cm}\!=\!\frac{n!\left(n(C_n^\lambda(x))^2\!+\!(2
\lambda\!+\!n\!-\!1)(C_{n-1}^\lambda(x))^2\!-\!x(2
\lambda\!+\!2n\!-\!1)C_{n}^\lambda(x)C_{n-1}^\lambda(x)\right)}
{2(1\!-\!x^2)\Gamma(2\lambda-1+n)}.
\end{eqnarray}
The conversion to the Ferrers functions of the first 
kind is accomplished using
\eqref{CnmuFerP}. This completes the proof.
\end{proof}
}

\noindent \boro{
Using \eqref{FerPFerQ}, the above Christoffel-Darboux formulas
for the Ferrers function of the first kind can be expressed 
in terms of the Ferrers function of the second kind.
}

\boro{
\begin{cor}
Let $n\in \mathbb N_0$, $x,x'\in(-1,1)$. Then
\begin{eqnarray}
&&\hspace{0.0cm}\sum_{k=0}^{n-1}
\frac{(n+k)k!}
{(2n+k-1)!}
{\sf Q}_{k+n-\frac12}^{n-\frac12}(x)
{\sf Q}_{k+n-\frac12}^{n-\frac12}(x')\nonumber
\\
&&\hspace{2cm}=
\frac
{n!}
{2(3n-2)!(x-x')}
\left(
{\sf Q}_{2n-\frac12}^{n-\frac12}(x)
{\sf Q}_{2n-\frac32}^{n-\frac12}(x')
-
{\sf Q}_{2n-\frac32}^{n-\frac12}(x)
{\sf Q}_{2n-\frac12}^{n-\frac12}(x')
\right),\label{CD3}\\
&&\hspace{-0.0cm}\sum_{k=0}^{n-1}
\frac{(n+k)k!}
{(2n+k-1)!}
\left({\sf Q}_{k+n-\frac12}^{n-\frac12}(x)\right)^2
\nonumber\\
&&\hspace{2cm}=
\frac{n!}
{2(3n-2)!(1-x^2)}
\biggl(
n
\left({\sf Q}_{2n-\frac12}^{n-\frac12}(x)
\right)^2
+(3n-1)\!\left(
{\sf Q}_{2n-\frac32}^{n-\frac12}(x)
\right)^2
\nonumber\\
&&\hspace{8.5cm}
-x(4n\!-\!1)
{\sf Q}_{2n-\frac32}^{n-\frac12}(x)
{\sf Q}_{2n-\frac12}^{n-\frac12}(x)
\biggr).\label{CD4}
\end{eqnarray}
\end{cor}
}
\boro{
\begin{proof}
Let $\mu=n-\frac12$ in \eqref{CD1}, \eqref{CD2}, respectively
using \eqref{FerPFerQ} completes the proof.
\end{proof}
}

%\poro{
%\[
%\sum_{k=0}^n \frac{1}{2^kk!}
%H_k(x)H_k(x')=
%\frac{H_{n+1}(x)H_n(x')
%-H_n(x)H_{n+1}(x')}{2^{n+1}n!(x-x')}
%\]
%}

\subsection{\boro{Poisson kernels}}

\boro{
The Poisson kernel $K_t(x,y)$ of an orthogonal
polynomial $p_n(x;{\bf a})$, where
${\bf a}$ is a set of parameters that the orthogonal
polynomal depends upon is given by
\begin{equation}
K_t(x,y;{\bf a})=\sum_{n=0}^\infty 
\frac{t^n}{h_n}p_n(x;{\bf a})p_n(y;{\bf a}),
\end{equation}
where $h_n$ \eqref{orthinner}
is the norm of the polynomial. In this subsection
we start with the Poisson kernel of Gegenbauer
polynomials and some companion identities to 
derive addition theorems for the Ferrers function
of the first kind.
}
%\medskip
\boro{
\begin{thm}
\label{PoissonGegFerP}
Let $t,\lambda\in{\mathbb C}$, 
$|t|<1$, 
%$m\in{\mathbb Z}$, 
$x,y\in(-1,1)$,
$x=\cos\theta$, $y=\cos\theta'$. Then
\begin{equation}
\sum_{n=0}^\infty
\frac{(2\lambda+1)_nt^n}
{n!}
{\sf P}_{n+\lambda}^{-\lambda}(x)
{\sf P}_{n+\lambda}^{-\lambda}(y)
=\frac{Q_{\lambda-\frac12}(\chi)}{\pi
\Gamma(2\lambda+1)t^{\lambda+\frac12}
\sqrt{\sin\theta\sin\theta'}},
\end{equation}
where 
\begin{equation}
\chi:=\frac{1+t^2-2t\cos\theta\cos\theta'}
{2\sin\theta\sin\theta'}.
\end{equation}
\end{thm}
}
%\rcoro{What is $m$ here?}
\boro{
\begin{proof}
Start with the Poisson kernel for the Gegenbauer polynomials
\cite[(19)]{Maier18}
\begin{eqnarray}
\sum_{n=0}^\infty
\frac{n!t^n}{(2\lambda)_n}
C_n^\lambda(x)C_n^\lambda(y)
&=&
\frac{1}{(1+t^2-2t\cos\theta\cos\phi)^{\lambda}}
\hyp21{\frac{\lambda}{2},\frac{\lambda+1}{2}}
{\lambda\!+\!\frac12}
{
\frac{1}{\chi^2}
%\frac{4t^2\sin^2\theta\sin^2\phi}
%{(1\!+\!t^2\!-\!2t\cos\theta\cos\phi)^2}
}\nonumber\\
&=&
\frac{\Gamma(\lambda+\frac12)Q_{\lambda-1}(\chi)}
{\sqrt{\pi}\Gamma(\lambda)t^\lambda(1-x^2)^{\frac
{\lambda}{2}}(1-y^2)^{\frac{\lambda}{2}}},
\label{PoissonGegC}
\end{eqnarray}
then replacing $\lambda\mapsto\lambda+\frac12$, and 
using \eqref{CnmuFerP} completes the proof.
\end{proof}
}

\boro{
\begin{rem}
Note that for $\lambda=m$, and setting $n\mapsto n+m$, Theorem 
\ref{PoissonGegFerP} specializes to the second addition theorem 
for spherical harmonics presented in 
\cite[(2.4)]{CRTB}
\begin{equation}
\sum_{n=|m|}^\infty
t^n\frac{(n-m)!}{(n+m)!}
{\sf P}_n^m(\cos\theta)
{\sf P}_n^m(\cos\theta')
=
\frac{Q_{m-\frac12}(\chi)}{\pi
\sqrt{t\sin\theta\sin\theta'}},
\end{equation}
after letting $t=r_</r_>$, where $r_\lessgtr:={\min 
\atop \max}\{r,r^\prime\}$ with $r,r^\prime\in[0,\infty)$.
\end{rem}
}

\boro{
\begin{thm}
\label{PoissonGegFerP2}
Let $t,\lambda\in{\mathbb C}$, 
$|t|<1$,
$x,y\in(-1,1)$,
$x=\cos\theta$, $y=\cos\theta'$. Then
\begin{eqnarray}
&&\hspace{-2cm}\sum_{n=0}^\infty
\frac{(\lambda+n+\frac12)(2\lambda+1)_nt^n}
{n!}
{\sf P}_{n+\lambda}^{-\lambda}(x)
{\sf P}_{n+\lambda}^{-\lambda}(y)
\nonumber\\
&&\hspace{1cm}=-\frac{1-t^2}{2\pi 
\Gamma(2\lambda+1)t^{\lambda+\frac32}
(1-x^2)^\frac34(1-y^2)^\frac34}
\frac{Q_{\lambda-\frac12}^1(\chi)}{\sqrt{\chi^2-1}}
.
\end{eqnarray}
\end{thm}
}
\boro{
\begin{proof}
Start with the companion Poisson kernel for
Gegenbauer polynomials \cite[(18)]{Maier18}
\begin{eqnarray}
&&\hspace{-1cm}\sum_{n=0}^\infty
\frac{\lambda+n}{\lambda}
\frac{n!t^n}{(2\lambda)_n}
C_n^\lambda(x)
C_n^\lambda(y)
\!=\!
\frac{1-t^2}{(1+t^2-2t\cos\theta\cos\phi)^{\lambda+1}}
\hyp21{\frac{\lambda+1}{2},\frac{\lambda+2}{2}}
{\lambda\!+\!\frac12}
{
\frac{1}{\chi^2}
%\frac{4t^2\sin^2\theta\sin^2\phi}
%{(1\!+\!t^2\!-\!2t\cos\theta\cos\phi)^2}
}\nonumber\\
&&\hspace{3.7cm}=
-\frac{\Gamma(\lambda+\frac12)(1-t^2)}
{\sqrt{\pi}\Gamma(\lambda+1)t^{\lambda+1}(1-x^2)^{\frac
{\lambda+1}{2}}(1-y^2)^{\frac{\lambda+1}{2}}}
\frac{Q_{\lambda-1}^1(\chi)}
{\sqrt{\chi^2-1}},
\label{PoissonKerC2}
\end{eqnarray}
then replacing $\lambda\mapsto\lambda+\frac12$, 
and using \eqref{CnmuFerP} completes the proof.
\end{proof}
}

\boro{
\begin{thm}
\label{PoissonGegFerP3}
Let $t,\lambda\in{\mathbb C}$, 
$|t|<1$,
$x,y\in(-1,1)$,
$x=\cos\theta$, $y=\cos\theta'$. Then
\begin{eqnarray}
&&\hspace{-0.5cm}\sum_{n=0}^\infty
\frac{(\lambda+n+\frac12)(\lambda+n+\frac32)(2\lambda+1)_nt^n}
{n!}
{\sf P}_{n+\lambda}^{-\lambda}(x)
{\sf P}_{n+\lambda}^{-\lambda}(y)
\nonumber\\
&&\hspace{0.05cm}=\frac{1}{\pi t^{\lambda-\frac12} \Gamma(2\lambda+1)
(1-x^2)^\frac34(1-y^2)^\frac34}
\left(
\frac{Q_{\lambda-\frac12}^1(\chi)}
{\sqrt{\chi^2-1}}
+
\frac{(1-t^2)^2Q_{\lambda-\frac12}^2(\chi)}
{4t^3\sqrt{1-x^2}\sqrt{1-y^2}(\chi^2-1)}
\right).
\end{eqnarray}
\end{thm}
}

\boro{
\begin{proof}
Start with 
Theorem \ref{PoissonGegFerP2}, and put all powers
of $t$ on the left-hand side of the equation and
differentiate with respect to $t$ using
\[
\frac{\mathrm d}{{\mathrm d}t}
\left(
\frac{t^{\lambda+n+\frac32}}{1-t^2}\right)
=\frac{2t^{\lambda+n+\frac52}}{(1-t^2)^2}
+\frac{(\lambda+n+\frac32)t^{\lambda+n+\frac12}}
{1-t^2},
\]
\[
\frac{{\partial}\chi}{{\partial}t}
=\frac{-(1-t^2)}{2t^2\sqrt{1-x^2}\sqrt{1-y^2}},
\]
and the following 
derivative formula for the associated Legendre function 
of the second kind (cf. \cite[Remark 4]{CohlCostasSantos2020})
\[
\frac{{\mathrm d}}
{{\mathrm d}\chi}
\frac{Q_{\lambda-\frac12}^m(\chi)}
{(1-\chi^2)^{\frac{m}{2}}}
=
\frac{Q_{\lambda-\frac12}^{m+1}(\chi)}
{(\chi^2-1)^{\frac{m+1}{2}}},
\]
for $m\in{\mathbb N}_0$.
Then applying Theorem \ref{PoissonGegFerP2}
again and the recurrence relation
\[
Q_{\lambda-\frac12}^{\mu+1}(\chi)
=\frac{1}{\sqrt{\chi^2-1}}
\left(
-(\lambda+\mu+\tfrac12)\chi Q_{\lambda-\frac12}^\mu(\chi)
+(\lambda-\mu+\tfrac12)Q_{\lambda+\frac12}^\mu(\chi)
\right),
\]
completes the proof.
\end{proof}
}

\boro{
\begin{cor}
\label{PoissonGeg1}
Let $t,\lambda\in{\mathbb C}$, $|t|<1$,
$x,y\in(-1,1)$, $x=\cos\theta$, $y=\cos\theta'$. Then
\begin{eqnarray}
&&\hspace{-0.55cm}\sum_{n=0}^\infty
(\lambda+n)(\lambda+n+1)
\frac{n!t^n}{(2\lambda)_n}
C_n^\lambda(x)
C_n^\lambda(y)\nonumber\\
&&\hspace{0.05cm}=\frac{\Gamma(\lambda+\frac12)}{\sqrt{\pi} 
t^{\lambda-1} \Gamma(\lambda)
(1-x^2)^{\frac{\lambda+1}{2}}(1-y^2)^{\frac{\lambda+1}{2}}}
\left(
\frac{Q_{\lambda-1}^1(\chi)}
{\sqrt{\chi^2-1}}
+
\frac{(1-t^2)^2Q_{\lambda-1}^2(\chi)}
{4t^3\sqrt{1-x^2}\sqrt{1-y^2}(\chi^2-1)}
\right).
\end{eqnarray}
\end{cor}
}
\begin{proof}
Using \eqref{CnmuFerP} in Theorem \ref{PoissonGegFerP3},
completes the proof.
\end{proof}

\boro{
\begin{cor}
\label{alggencorGeg}
Let $t,\lambda\in{\mathbb C}$, 
$|t|<1$,
$x,y\in(-1,1)$,
$x=\cos\theta$, $y=\cos\theta'$. Then, one has the following
algebraic generating functions for the Gegenbauer polynomials:
\begin{eqnarray}
&&\label{alggen0}
\sum_{n=0}^\infty t^n C_n^\lambda(x)
=\frac{1}{(1+t^2-2tx)^{\lambda}},\\
&&\label{alggen1}
\sum_{n=0}^\infty(\lambda+n)t^n C_n^\lambda(x)
=\frac{\lambda(1-t^2)}{(1+t^2-2tx)^{\lambda+1}},\\
&&\label{alggen2}
\sum_{n=0}^\infty(\lambda+n)(\lambda+n+1)t^n C_n^\lambda(x)
=\frac{\lambda(\lambda+1)(1-t^2)^2}{(1+t^2-2tx)^{\lambda+2}}
-\frac{2\lambda t^2}{(1+t^2-2tx)^{\lambda+1}}.
\end{eqnarray}
\end{cor}
}
\begin{proof}
Consider the limit as $y\to 1^{-}$ as follows.
Let $y=1-\tfrac12\epsilon^2$ with $\epsilon>0$, such that
$\epsilon\ll 1$. Then $\sqrt{1-y^2}\sim\epsilon$
and 
\[
\chi\sim\frac{1+t^2-2tx}{2t\sqrt{1-x^2}\epsilon},
\]
as $\epsilon\to 0^{+}$. 
Note that
\cite[Table 18.6.1]{NIST:DLMF}
$C_n^\lambda(1)=(2\lambda)_n/n!$.
Applying this and the above estimates 
in
\eqref{PoissonGegC}, 
\eqref{PoissonKerC2}
and Corollary \ref{PoissonGeg1}, then taking the limit as
$\epsilon\to{0}^{+}$, completes the proof.
\end{proof}

\boro{
\begin{cor}
Let $t,\lambda\in{\mathbb C}$, $|t|<1$,
$x,y\in(-1,1)$, $x=\cos\theta$, $y=\cos\theta'$. Then 
\begin{eqnarray}
&&\label{Palggen0}
\sum_{n=0}^\infty \frac{t^n(2\lambda+1)_n}{n!} 
{\sf P}_{n+\lambda}^{-\lambda}(x)
=\frac{(1-x^2)^{\frac{\lambda}{2}}}
{2^\lambda\Gamma(\lambda+1)(1+t^2-2tx)^{\lambda+\frac12}},\\
&&\label{Palggen1}
\sum_{n=0}^\infty(\lambda+n+\tfrac12) \frac{t^n(2\lambda+1)_n}{n!} 
{\sf P}_{n+\lambda}^{-\lambda}(x)=\frac{(\lambda+\tfrac12)(1-t^2)
(1-x^2)^{\frac{\lambda}{2}}}{2^\lambda\Gamma(\lambda+1)(1+t^2
-2tx)^{\lambda+\frac32}},\\
&&\label{Palggen2}
\sum_{n=0}^\infty(\lambda+n+\tfrac12)(\lambda+n+\tfrac32) 
\frac{t^n(2\lambda+1)_n}{n!} {\sf P}_{n+\lambda}^{-\lambda}(x)
\nonumber\\
&&\hspace{3cm}=\frac{(1-x^2)^{\frac{\lambda}{2}}(\lambda
+\tfrac12)}{2^\lambda\Gamma(\lambda+1)}\left(
\frac{(\lambda+\tfrac32)(1-t^2)^2}{(1+t^2-2tx)^{\lambda+\tfrac52}}
-\frac{2t^2}{(1+t^2-2tx)^{\lambda+\tfrac32}}\right).
\end{eqnarray}
\end{cor}
}
\boro{
\begin{proof}
Starting with Corollary \ref{alggencorGeg}
and replacing the Gegenbauer polynomials using 
\eqref{CnmuFerP} completes the proof.
\end{proof}
}

\medskip 

%\subsection{\moro{Roberto's playground}}
%\moro{Moreover, the following factorization also holds true:
%\begin{equation}\label{factgesp}
%{\sf P}_{k-N}^{N}(x)=\frac{(1)_k}{(1-2\mu)_k}
%{\sf P}_{-N}^{N}(x)(x^2-1)^N C_{k-2N}^{\frac 12+N}(x).
%\end{equation}
%\moro{Another interesting property of these polynomials is 
%the following \cite[(14.3.21)]{NIST:DLMF}: 
%}
%}

\subsection{\boro{Closure relations}}

\boro{
Closure relations are infinite
series representations of the
Dirac delta distribution.
Let $n,m\in \mathbb N_0$,
the closure relations
are obtained from a complete 
set of orthogonal functions $\psi_n(x)$ 
on $(a,b)$ with a weight function $w(x)$,
with respect to an inner product defined as
\begin{equation}
\langle f,g\rangle:=\int_a^b f(x) 
%\overline{g(x)} 
{g(x)} 
w(x) \,{\mathrm d}x,
\end{equation}
so that the moments $(m_n)$ are bounded, where
\begin{equation}
m_n:=\langle x^n,1\rangle.
\end{equation}
Then, the integral orthogonality relation is given by
\begin{equation}
\label{orthinner}
\langle\psi_n,\psi_m\rangle=
\int_{a}^b\psi_n(x)\psi_m(x)w(x)\,{\mathrm d}x=h_n\delta_{n,m}.
\end{equation}
}

\boro{
\begin{prop}
Let $x,x'\in[a,b]$ and
the Poisson kernel is
non-negative.
Then, one has the following closure relation:
\begin{equation}
\sum_{n=0}^\infty
\frac{w(x)}{h_n}
\psi_n(x)\psi_n(x')=\delta(x-x'),
%\quad
%\sum_{n=0}^\infty
%\frac{\psi_n(x)\psi_n(x')}{h_n}=
%\frac{\delta(x-x')}{w(x)},
\label{Closure}
\end{equation}
where both the left- and right-hand sides should 
be treated as distributions.
\end{prop}
}

%\rcoro{Then, we insert the proof of the result anyway?}
\boro{
\begin{proof}{\!\bf(formal)}
%Then 
One may expand a
function $f\in L^2(a,b)$ in the infinite
series of the orthogonal functions as
\begin{equation}\label{fourierf}
f(x)=\sum_{n=0}^\infty a_n\psi_n(x).
\end{equation}
Now multiplying both sides of
this equation by
$w(x)\psi_m(x)$ and integrating
over the interval $(a,b)$
produces
\[
a_n=\frac{1}{h_n}
\int_a^b f(x)\psi_n(x)w(x)\,{\mathrm d}x.
\]
Multiplying the above expression by $\psi_n(x')$
and summing over all non-negative
integers produces 
\[
f(x')=\int_a^b f(x)
\left(\,
\sum_{n=0}^\infty
\frac{w(x)}{h_n}
\psi_n(x)\psi_n(x') \right)\,{\mathrm d}x
=\int_a^b f(x) \delta(x-x')\,{\mathrm d}x,
\]
where we have used the sifting property
of the Dirac delta distribution
\cite[(1.16.11)]{NIST:DLMF}.
This completes the proof.
See \cite[Theorem 2.1]{IsmailZhangZhou} for a 
different
proof.
\end{proof}
}

\boro{
The following is the closure relation for 
Ferrers function of the first kind. 
\begin{thm}
Let $x,x'\in(-1,1)$, $\mu\in\mathbb C$, $2\mu+1\not
\in-\mathbb N_0$. Then
\begin{equation}
\sum_{n=0}^\infty\frac{(\mu+n+\tfrac12)\Gamma(2\mu+n+1)}{n!}
{\sf P}_{n+\mu}^{-\mu}(x){\sf P}_{n+\mu}^{-\mu}(x')
=\delta(x-x').
\label{compPb}
\end{equation}
%or in terms of orthogonal Ferrers
%functions of the first kind
%\begin{equation}
%\label{compPa}
%\sum_{n=l}^\infty\frac{(\mu+n+\tfrac12)\Gamma(2\mu+l+n+1)}{(n-l)!}
%{\sf P}_{n+\mu}^{-\mu-l}(x){\sf P}_{n+\mu}^{-\mu-l}(x')
%=\delta(x-x').
%\end{equation}
\end{thm}
}
\boro{
\begin{proof}
Using the orthogonality relation for Gegenbauer polynomials
\eqref{orthogGeg} in order to obtain $h_n$ with \eqref{Closure}
produces
\begin{equation}
\sum_{n=0}^\infty \frac{(\mu+n)n!}{\Gamma(2\mu+n)}
C_n^\mu(x)C_n^\mu(x')=\frac{\pi \delta(x-x')}
{(1-x^2)^{\mu-\frac12} 2^{2\mu-1}\Gamma^2(\mu)}.
\end{equation}
Applying \eqref{CnmuFerP} twice completes the proof.
\end{proof}
}

\boro{
\begin{rem}
Note that in the limit as $n\to\infty$ the 
Christoffel-Darboux formula \eqref{CD1} becomes the 
closure relation \eqref{compPb}. Furthermore, the 
Poisson kernel can be obtained from the 
Christoffel-Darboux sum with each coefficient
multiplied by some $t^k$, with $|t|<1$.
This interrelation as well as the interrelation
with universality for orthogonal 
polynomials are quite intriguing
(see the excellent review article by 
Barry Simon entitled ``The Christoffel-Darboux Kernel'' \cite[pp.~295--335]{SimonPerspectives08}).
\end{rem}
}

\begin{cor}
Let $x\in(-1,1)$, $\lambda,\gamma\in\mathbb C$
such that $2\lambda-\gamma+2\not\in-{\mathbb N}_0$. Then
\begin{equation}
\sum_{n=0}^\infty
\frac{(\lambda+n+\frac12)(\gamma)_n(2\lambda+1)_n}
{n!(2\lambda-\gamma+2)_n}
{\sf P}_{n+\lambda}^{-\lambda}(x)
=\frac{\Gamma(2\lambda-\gamma+2)(1-x)^{\frac{\lambda}
{2}-\gamma}(1+x)^{\frac{\lambda}{2}}}{2^{\lambda-\gamma+1}
\Gamma(2\lambda+1)\Gamma(\lambda-\gamma+1)}.
\end{equation}
\end{cor}
\begin{proof}
Applying the closure relation
\eqref{compPb} to 
\eqref{thisintFerP}, multiplying
by the necessary factors
and integrating over $(-1,1)$ 
completes the proof.
\end{proof}

\boro{
\begin{rem}
It is interesting to see that using 
orthogonality for Jacobi polynomials 
\cite[Table 18.3.1]{NIST:DLMF}, the closure relation
for Jacobi polynomials is given by
\begin{equation}
\sum_{n=0}^\infty
\frac{(2n+\alpha+\beta+1)\Gamma(\alpha+\beta+n+1)n!}
{\Gamma(\alpha+n+1)\Gamma(\beta+n+1)}
P_n^{(\alpha,\beta)}(x)
P_n^{(\alpha,\beta)}(x')
=\frac{2^{\alpha+\beta+1}\delta(x-x')}
{(1-x)^\alpha(1+x)^\beta}.
\end{equation}
\end{rem}
}

\boro{
\begin{rem}
It should be noted that if one takes $\mu=0$ in 
\eqref{compPb} then one obtains the closure relation
for Legendre polynomials \eqref{Legpoly}
\cite[(1.17.22)]{NIST:DLMF}
\begin{equation}
\sum_{n=0}^\infty (n+\tfrac12)P_n(x)P_n(x')=\delta(x-x').
\end{equation}
On the other hand, using $\mu=\pm1/2$ in \eqref{compPb} 
produces the following closure relations for the Chebyshev 
polynomials of the first and second kinds 
\cite[Table 18.3.1]{NIST:DLMF}, 
namely:
\begin{eqnarray}
%&&\hspace{-5.5cm}
&&\label{ClosTn}\hspace{-5.5cm}\sum_{n=0}^\infty T_n(x)T_n(x')=\frac{\pi}{2}
(1-x^2)^\frac14(1-x'^2)^{\frac14}\,\delta(x-x'), \\
%&&\hspace{-5.5cm}
&&\hspace{-5.5cm}\sum_{n=0}^\infty U_n(x)U_n(x')=\frac{\pi\delta(x-x')}
{2(1-x^2)^\frac14(1-x'^2)^{\frac14}}. \label{ClosUn}
\end{eqnarray}
\end{rem}
}

\noindent \boro{Using \eqref{QQorthog} we can obtain a 
closure relation for the Ferrers function of the second kind.
}
\boro{
\begin{cor}
Let $n\in \mathbb N_0$, $x,x'\in(-1,1)$. Then
%\begin{eqnarray}
\begin{equation}
%&&\hspace{-5cm}\sum_{k=l}^\infty
%\frac{(n+k+1)(k-l)!}
%{(2n+k+l+1)!}
%{\sf Q}_{k+n+\frac12}^{l+n+\frac12}(x)
%{\sf Q}_{k+n+\frac12}^{l+n+\frac12}(x')
%=\frac{\pi^2}{4}\delta(x-x'),
%\\
%&&\hspace{-5cm}
\sum_{k=0}^\infty
\frac{(n+k)k!}{(2n+k-1)!}
{\sf Q}_{k+n-\frac12}^{n-\frac12}(x)
{\sf Q}_{k+n-\frac12}^{n-\frac12}(x')
=\frac{\pi^2}{4}\delta(x-x').
\label{closQb}
\end{equation}%\end{eqnarray}
\end{cor}
}
\boro{
\begin{proof}
From \eqref{QQorthog} we obtain 
$h_n=\pi^2(2n+k-1)!/(4(n+k)k!)$, then using \eqref{Closure} 
with $w(x)=1$ completes the proof.
\end{proof}
}

\boro{The $n=0,1$ specializations of \eqref{closQb} correspond
with the closure relation for the Chebyshev polynomials 
of the first and second kinds \eqref{ClosTn}, \eqref{ClosUn}.}

\boro{
\begin{cor}
Let $x, x'\in(-1,1)$, $x=\cos\theta$, $x'=\cos\theta'$. 
Then
\begin{eqnarray}
&&\hspace{-1cm}\sum_{k=0}^\infty
k^2
{\sf Q}_{k-\frac12}^{-\frac12}(x)
{\sf Q}_{k-\frac12}^{-\frac12}(x')
=
\frac{\pi}{2\sqrt{\sin\theta\sin\theta'}}
\sum_{k=0}^\infty
\cos(k\theta)\cos(k\theta')
=\frac{\pi^2}{4}\delta(x-x'),\\
&&\hspace{-1cm}\sum_{k=0}^\infty
{\sf Q}_{k+\frac12}^{\frac12}(x)
{\sf Q}_{k+\frac12}^{\frac12}(x')
=\frac{\pi}{2\sqrt{\sin\theta\sin\theta'}}
\sum_{k=1}^\infty\sin(k\theta)
\sin(k\theta')
=\frac{\pi^2}{4}\delta(x-x').
\end{eqnarray}
%which are equivalent to %\eqref{ClosTn},
%\eqref{ClosUn}.
\end{cor}
}
\boro{
\begin{proof}
Start with \eqref{closQb} with $n=0,1,$ and using \eqref{Qhalf},
\eqref{Qmhalf},
completes the proof.
\end{proof}
}

%\goro{The $n=0$ specialization of \eqref{closQb} corresponds
%with the closure relation for Chebyshev polynomials 
%of the first kind \eqref{ClosTn}.
%Need to finish this. I thought I did.
%Insert.}

%\noindent\poro{\bf [HSC: Note that I prefer to use the notation for Gegenbauer polynomials without the parentheses in the superscript. I am not alone in this preference. Tom Koornwinder also uses this notation, and also so did Gegenbauer himself,
%and I've been using it for
%years, so there you go. :)]} \moro{[RCS: I am ok with this. You are the expert.]}

\section*{Acknowledgements}
Much thanks to \boro{Adri Olde Daalhuis and Nico Temme} for valuable discussions.
The research of R.S.C.-S. was funded by Agencia Estatal de Investigación of Spain, grant number PGC-2018-096504-B-C33.

%%%\section*{References}
%\bibliographystyle{plain}
%\bibliography{refbib} % Howard bib yes!

\begin{thebibliography}{10}

\bibitem{AskeyRazban72}
R.~Askey and B.~Razban.
\newblock An integral for {J}acobi polynomials.
\newblock {\em Simon Stevin}, 46:165–169, 1972.

\bibitem{Cohl12pow}
H.~S. {Cohl}.
\newblock {Fourier, Gegenbauer and Jacobi expansions for a power-law
  fundamental solution of the polyharmonic equation and polyspherical addition
  theorems}.
\newblock {\em Symmetry, Integrability and Geometry: Methods and Applications},
  9(042):26, 2013.

\bibitem{CohlCostasSantos2020}
H.~S. {Cohl} and R.~S. {Costas-Santos}.
\newblock {Multi-integral representations for associated Legendre and Ferrers
  functions}.
\newblock {\em Symmetry}, 12(1598):22, 2020.

\bibitem{CohlCostasWakhare2019}
H.~S. {Cohl}, R.~S. {Costas-Santos}, and T.~V. {Wakhare}.
\newblock {On a generalization of the Rogers generating function}.
\newblock {\em Journal of Mathematical Analysis and Applications},
  475(2):1019--1043, 2019.

\bibitem{Cohletal2021}
H.~S. {Cohl}, J.~{Park}, and H.~{Volkmer}.
\newblock Gauss hypergeometric representations of the ferrers function of the
  second kind.
\newblock {\em SIGMA. Symmetry, Integrability and Geometry. Methods and
  Applications}, 17:Paper No. 053, 33, 2021.

\bibitem{CRTB}
H.~S. Cohl, A.~R.~P. Rau, J.~E. Tohline, D.~A. Browne, J.~E. Cazes, and E.~I.
  Barnes.
\newblock Useful alternative to the multipole expansion of $1/r$ potentials.
\newblock {\em Physical Review A: Atomic and Molecular Physics and Dynamics},
  64(5):052509, Oct 2001.

\bibitem{CohlVolkmerDefInt}
H.~S. Cohl and H.~{Volkmer}.
\newblock Definite integrals using orthogonality and integral transforms.
\newblock {\em Symmetry, Integrability and Geometry:~Methods and Applications,
  Special Issue on Superintegrability, Exact Solvability, and Special
  Functions}, 8, 2012.

\bibitem{NIST:DLMF}
{\it NIST Digital Library of Mathematical Functions}.
\newblock \href{http://dlmf.nist.gov/}{\bf\tt\normalsize http://dlmf.nist.gov},
  Release 1.1.2 of 2021-06-15.
\newblock F.~W.~J. Olver, A.~B. {Olde Daalhuis}, D.~W. Lozier, B.~I. Schneider,
  R.~F. Boisvert, C.~W. Clark, B.~R. Miller, B.~V. Saunders, H.~S. Cohl, and
  M.~A. McClain, eds.

\bibitem{ErdelyiTII}
A.~Erd{\'e}lyi, W.~Magnus, F.~Oberhettinger, and F.~G. Tricomi.
\newblock {\em Tables of Integral Transforms. {V}ol. {II}}.
\newblock McGraw-Hill Book Company, Inc., New York-Toronto-London, 1954.

\bibitem{Grad}
I.~S. Gradshteyn and I.~M. Ryzhik.
\newblock {\em Table of {I}ntegrals, {S}eries, and {P}roducts}.
\newblock Elsevier/Academic Press, Amsterdam, seventh edition, 2007.

\bibitem{IsmailZhangZhou}
M.~E.~H. {Ismail}, R.~{Zhang}, and K.~{Zhou}.
\newblock {$q$-Fractional Askey-Wilson Integrals and Related Semigroups of
  Operators}.
\newblock {\em {\tt in preparation}}, 2021.

\bibitem{Koekoeketal}
R.~Koekoek, P.~A. Lesky, and R.~F. Swarttouw.
\newblock {\em Hypergeometric orthogonal polynomials and their
  {$q$}-analogues}.
\newblock Springer Monographs in Mathematics. Springer-Verlag, Berlin, 2010.
\newblock With a foreword by Tom H. Koornwinder.

\bibitem{Maier18}
R.~S. Maier.
\newblock Algebraic generating functions for {G}egenbauer polynomials.
\newblock In {\em Frontiers in orthogonal polynomials and {$q$}-series},
  volume~1 of {\em Contemp. Math. Appl. Monogr. Expo. Lect. Notes}, pages
  425--444. World Sci. Publ., Hackensack, NJ, 2018.

\bibitem{Olver}
F.~W.~J. Olver.
\newblock {\em {Asymptotics and Special Functions}}.
\newblock AKP Classics. A K Peters Ltd., Wellesley, MA, 1997.
\newblock Reprint of the 1974 original [Academic Press, New York].

\bibitem{salara}
J.~F. S\'{a}nchez-Lara.
\newblock On the {S}obolev orthogonality of classical orthogonal polynomials
  for non standard parameters.
\newblock {\em Rocky Mountain J. Math.}, 47(1):267--288, 2017.

\bibitem{SimonPerspectives08}
B.~Simon.
\newblock The {C}hristoffel-{D}arboux kernel.
\newblock In {\em Perspectives in partial differential equations, harmonic
  analysis and applications, A Volume in Honor of Vladimir G.~Maz'ya's 70th
  Birthday}, volume~79, pages 295--335, Providence, RI, USA, 2008. American
  Mathemtical Society.

\bibitem{Szmy2}
R.~{Szmytkowski}.
\newblock {A note on parameter derivatives of classical orthogonal
  polynomials}.
\newblock {\em arXiv:0901.2639}, 2018.

\end{thebibliography}

\def\cprime{$'$} \def\dbar{\leavevmode\hbox to 0pt{\hskip.2ex \accent"16\hss}d}

\end{document}